\documentclass[a4paper, 11pt]{amsart}    
\usepackage{latexsym, amsmath, amsthm, amssymb, setspace, verbatim}
\usepackage[all]{xy}
\usepackage[utf8]{inputenc} \usepackage[T1]{fontenc} \usepackage{lmodern}
\usepackage{hyperref}
\usepackage{enumitem}

\title[Equivariant Kirchberg-Phillips-type absorption]{Equivariant Kirchberg-Phillips-type absorption for amenable group actions}
\author{Gábor Szabó}
\address{Department of Mathematical Sciences, University of Copenhagen \phantom{----------}\linebreak \text{}\hspace{3.0mm} Universitetsparken 5, DK-2100 Copenhagen {\O}, Denmark.}
\email{gabor.szabo@math.ku.dk}
\thanks{\emph{Supported by:} SFB 878 \emph{Groups, Geometry and Actions}, EPSRC grant EP/N00874X/1, the Danish National Research Foundation through the \emph{Centre for Symmetry and Deformation} (DNRF92), and the European Union's Horizon 2020 research and innovation programme under the Marie Sklodowska-Curie grant agreement 746272.}

\numberwithin{equation}{section}

\begin{document}

\renewcommand\matrix[1]{\left(\begin{array}{*{10}{c}} #1 \end{array}\right)}  
\newcommand\set[1]{\left\{#1\right\}}  

\newcommand{\IC}[0]{\mathbb{C}}
\newcommand{\IN}[0]{\mathbb{N}}
\newcommand{\IQ}[0]{\mathbb{Q}}
\newcommand{\IR}[0]{\mathbb{R}}
\newcommand{\IT}[0]{\mathbb{T}}
\newcommand{\IZ}[0]{\mathbb{Z}}

\newcommand{\CC}[0]{\mathcal{C}}
\newcommand{\CD}[0]{\mathcal{D}}
\newcommand{\CK}[0]{\mathcal{K}}
\newcommand{\CL}[0]{\mathcal{L}}
\newcommand{\CM}[0]{\mathcal{M}}
\newcommand{\CO}[0]{\mathcal{O}}
\newcommand{\CU}[0]{\mathcal{U}}
\newcommand{\CZ}[0]{\mathcal{Z}}

\renewcommand{\phi}[0]{\varphi}
\newcommand{\eps}[0]{\varepsilon}

\newcommand{\id}[0]{\operatorname{id}}		
\renewcommand{\sp}[0]{\operatorname{Sp}}		
\newcommand{\eins}[0]{\mathbf{1}}			
\newcommand{\diag}[0]{\operatorname{diag}}
\newcommand{\ad}[0]{\operatorname{Ad}}
\newcommand{\ev}[0]{\operatorname{ev}}
\newcommand{\fin}[0]{{\subset\!\!\!\subset}}
\newcommand{\Aut}[0]{\operatorname{Aut}}
\newcommand{\dimrok}[0]{\dim_{\mathrm{Rok}}}
\newcommand{\dst}[0]{\displaystyle}
\newcommand{\cstar}[0]{\ensuremath{\mathrm{C}^*}}
\newcommand{\dist}[0]{\operatorname{dist}}
\newcommand{\ann}[0]{\operatorname{Ann}}
\newcommand{\cc}[0]{\simeq_{\mathrm{cc}}}
\newcommand{\scc}[0]{\simeq_{\mathrm{scc}}}
\newcommand{\vscc}[0]{\simeq_{\mathrm{vscc}}}

\newtheorem{theorem}{Theorem}[section]	
\newtheorem{cor}[theorem]{Corollary}
\newtheorem{lemma}[theorem]{Lemma}
\newtheorem{prop}[theorem]{Proposition}

\newcounter{theoremintro}
\newtheorem{theoremi}[theoremintro]{Theorem}
\renewcommand*{\thetheoremintro}{\Alph{theoremintro}}

\theoremstyle{definition}
\newtheorem{defi}[theorem]{Definition}
\newtheorem{nota}[theorem]{Notation}
\newtheorem*{notae}{Notation}
\newtheorem{rem}[theorem]{Remark}
\newtheorem*{reme}{Remark}
\newtheorem{example}[theorem]{Example}
\newtheorem*{examplee}{Example}

\begin{abstract} 
We show an equivariant Kirchberg-Phillips-type absorption theorem for pointwise outer actions of discrete amenable groups on Kirchberg algebras with respect to natural model actions on the Cuntz algebras $\CO_\infty$ and $\CO_2$. 
This generalizes results known for finite groups and poly-$\IZ$ groups.
The model actions are shown to be determined, up to strong cocycle conjugacy, by natural abstract properties, which are verified for some examples of actions arising from tensorial shifts.
We also show the following homotopy rigidity result, which may be understood as a precursor to a general Kirchberg-Phillips-type classification theory:\ If two outer actions of an amenable group on a unital Kirchberg algebra are equivariantly homotopy equivalent, then they are conjugate. 
This marks the first \cstar-dynamical classification result up to cocycle conjugacy that is applicable to actions of all amenable groups.
\end{abstract}

\maketitle

\tableofcontents


\section*{Introduction}

The problem of classifying group actions has a long history within the theory of operator algebras.
A typical example is the Connes-Haagerup classification of injective factors \cite{Connes76, Haagerup87}, which initially involved a classification of cyclic group actions \cite{Connes75, Connes77}.
This has sparked a lot of work towards classifying group actions on von Neumann algebras, the highlight of which is the classification of actions of countable, amenable groups on injective  factors \cite{Connes75, Connes77, Jones80, Ocneanu85, SutherlandTakesaki89, KawahigashiSutherlandTakesaki92, KatayamaSutherlandTakesaki98, Masuda07} achieved by many hands.
See also \cite{Popa95, Popa10} for Popa's subfactor approach to these kinds of results.

Drawing the comparison to von Neumann algebras, the classification of group actions on \cstar-algebras has been less successful.
This is not least because of $K$-theoretical obstructions that do not appear in the context of von Neumann algebras.
Let us summarize the current state for \cstar-algebras, at least for discrete groups.
In his pioneering work \cite{Kishimoto95, Kishimoto96, Kishimoto98, Kishimoto98II}, Kishimoto has outlined a method to use the Rokhlin property in order to classify single automorphisms (i.e.~$\IZ$-actions) on certain inductive limit \cstar-algebras such as A$\IT$ algebras.
Finite group actions with the Rokhlin property were successfully classified by Izumi \cite{Izumi04, Izumi04II} on classifiable classes of \cstar-algebras.
The seminal technique of the so-called Evans-Kishimoto intertwining, introduced in \cite{EvansKishimoto97}, plays a key role in this regard. 
Classification via this method has been developed further by work of Matui \cite{Matui10} for AH algebras, Nakamura \cite{Nakamura00} for Kirchberg algebras, and Lin \cite{Lin15} for TAI algebras.
It is noteworthy that the Evans-Kishimoto intertwining technique has given some impetus back towards actions on von Neumann algebras, for example in Masuda's unified treatment \cite{Masuda07} of amenable group actions on injective factors.
Recently, Matui-Sato have announced a classification of single automorphisms with the weak Rokhlin property on classifiable \cstar-algebras; with the present technology available \cite{GongLinNiu15, ElliottGongLinNiu15, TikuisisWhiteWinter15}, this means that the involved \cstar-algebras are separable, unital, simple UCT \cstar-algebras with finite nuclear dimension.
The first abstract classification of single automorphisms on stably projectionless \cstar-algebras was in turn obtained recently by Nawata \cite{Nawata17}.
The aforementioned techniques and results for $\IZ$-actions are being pushed to $\IZ^d$-actions and even beyond, see \cite{Nakamura99, KatsuraMatui08, Matui08, Matui10, Matui11, IzumiMatui10, Izumi12OWR}. 

The \cstar-algebraic results cited above are impressive achievements, but they are all in some way underpinned by the Rokhlin property.
This causes the present theory to have some limitations in its applicability; for example, it is well-known that for finite group actions on Kirchberg algebras, the Rokhlin property is a fairly special occurrence, and actions satisfying it might not even exist on \cstar-algebras like $\CO_\infty$.
But even if one were to accept restrictions on the class of actions to consider, there are to my knowledge no sensible \cstar-algebraic candidates for the Rokhlin property for actions of groups that are not residually finite; see \cite[Section 3]{Izumi10}.
The only obvious exception is Matui-Sato's weak Rokhlin property \cite{MatuiSato12, MatuiSato14} originating in \cite{Sato10}.

Comparing the state of the theory of actions on \cstar-algebras vs.\ the aforementioned developments in von Neumann algebras, this calls for new ideas to handle more general actions of any amenable group.
Given the increasingly clear parallels between concepts for von Neumann algebras and for \cstar-algebras, which has driven the \cstar-algebraic theory to its recent breakthroughs \cite{BBSTWW, GongLinNiu15, ElliottGongLinNiu15, TikuisisWhiteWinter15}, one would expect that there is much more to discover about amenable group actions on \cstar-algebras.
In the long run, it is not far-fetched to expect that on classifiable \cstar-algebras, certain classes of amenable group actions could be classified with an invariant incorporating $K$-theory and traces (and likely more), matching the analogy between the well-established classification of injective factors and the recent breakthroughs in the classification of \cstar-algebras.  

Motivated by the fundamental importance of strongly self-absorbing \cstar-algebras in the Elliott program, an appropriate substitute for the Rokhlin property could be the absorption of certain strongly self-absorbing actions in the sense of \cite{Szabo16ssa, Szabo16ssa2, Szabo17ssa3}.
Tensorial stability has recurringly played an important role since Connes' proof of the uniqueness of the injective II${}_1$-factor, and has in part given rise to the idea of requiring a \cstar-algebraic replacement for the McDuff property for the purpose of classification.
This vague idea is fleshed out in the form of Toms-Winter's concept of strongly self-absorbing \cstar-algebras \cite{TomsWinter07}, which ought to be considered as the closest analogues of the hyperfinite II${}_1$-factor in this regard.
Indeed, for the Jiang-Su algebra $\CZ$, which is the initial object in the category of strongly self-absorbing \cstar-algebras \cite{Winter11}, the most general analogue of the McDuff property is given by $\CZ$-stability and plays a major role in the Elliott program through the Toms-Winter regularity conjecture \cite{ElliottToms08, WinterZacharias10, Winter12}.

The theory of strongly self-absorbing \cstar-dynamical systems was in turn fleshed out in previous work of the author, and forms an equivariant generalization of Toms-Winter's original concept;
the primary motivation being that they ought to play a similarly important role in equivariant classification as the strongly self-absorbing \cstar-algebras do within the Elliott classification program for simple, nuclear \cstar-algebras.
Although little evidence for this could be given in the initial theory-building work carried out in \cite{Szabo16ssa, Szabo16ssa2}, the most recent \cite{Szabo17ssa3} of these articles allowed for a small breakthrough concerning the uniqueness of strongly outer actions on strongly self-absorbing \cstar-algebras, which is now verified for torsion-free abelian groups.
In contrast to this work, however, the approach has implicitly still relied on the Rokhlin property by building on results from \cite{Matui08, Matui11, IzumiMatui10}.
I am confident that this paper represents a major step towards further establishing the importance of strongly self-absorbing actions in the classification theory of group actions on \cstar-algebras, as well as a novel approach to extend the current theory's applicability beyond residually finite groups.

If one considers pointwise outer actions of amenable groups on Kirchberg algebras, then one might expect that equivariant Kirchberg-Phillips-type absorption results should be the first step towards the classification of such actions up to cocycle conjugacy;
 the latter would constitute a far-reaching equivariant generalization of the celebrated Kirchberg-Phillips classification. 
The classical Kirchberg-Phillips absorption results \cite{KirchbergPhillips00} assert that all Kirchberg algebras absorb the Cuntz algebra $\CO_\infty$ of infinitely many generators and are absorbed by the Cuntz algebra $\CO_2$ of two generators tensorially, which Phillips' approach \cite{Phillips00} to the classification of these algebras made use of; see also \cite{KirchbergC}. 
In this way, one can regard the Cuntz algebras $\CO_2$ and $\CO_\infty$ as cornerstones of the classification of Kirchberg algebras.
For finite group actions, existing equivariant Kirchberg-Phillips-type absorption results are due to Izumi \cite{Izumi04} and Goldstein-Izumi \cite{GoldsteinIzumi11}.
A variant of this has been announced in ongoing work by Izumi-Matui \cite{Izumi12OWR} for actions of poly-$\IZ$ groups.
Work in progress by Phillips suggests that one can build on such absorption theorems to classify (pointwise outer) actions of finite groups on Kirchberg algebras via equivariant $KK$-theory. 
In light of the discussion above, it is natural to expect that this type of classification could be possible for more general amenable groups.

For a given amenable group, one can construct natural examples of pointwise outer model actions on $\CO_\infty$ and $\CO_2$ as follows:

\begin{examplee}[see Examples \ref{the-model-Oinf} and \ref{the-model-O2}]
Let $G$ be a countable, discrete and amenable group. Since $\cstar(G)$ unitally embeds into $\CO_2$, there exists a unitary representation $v: G\to\CU(\CO_2)$ that induces a $*$-monomorphism on $\cstar(G)$. Fix some non-zero $*$-homomorphism $\iota:\CO_2\to\CO_\infty$. Then $u_g=\iota(v_g)+\eins-\iota(\eins)$ yields a unitary representation $u: G\to\CU(\CO_\infty)$ that still induces a $*$-monomorphism on $\cstar(G)$. Then we may consider the actions
\[
\gamma=\bigotimes_\IN\ad(u): G\curvearrowright\bigotimes_\IN\CO_\infty\cong\CO_\infty
\]
and
\[
\delta=\bigotimes_\IN\ad(v): G\curvearrowright\bigotimes_\IN\CO_2\cong\CO_2.
\]
\end{examplee}

As explained in \cite[Section 5]{Szabo16ssa}, these actions are models for the actions in the finite group case that fit into the known equivariant absorption theorem. They are moreover strongly self-absorbing actions. The approach that we follow in this paper is to take precisely these two model actions as candidates fitting into an equivariant absorption theorem for all amenable groups. The following are the main results.

\begin{theoremi}[see Corollary \ref{model-uniqueness} and Example \ref{the-model-O2}] \label{choice-independence-intro}
Up to strong cocycle conjugacy, the actions $\gamma$ and $\delta$ above do not depend on the choice of unitary representations.
\end{theoremi}

The following is the equivariant analogue of the Kirchberg-Phillips $\CO_\infty$-absorption theorem.

\begin{theoremi}[see Theorem \ref{model-absorption}] \label{model-absorption-intro}
Every pointwise outer cocycle action $(\alpha,u): G\curvearrowright A$ on a Kirchberg algebra is strongly cocycle conjugate to $(\alpha\otimes\gamma,u\otimes\eins): G\curvearrowright A\otimes\CO_\infty$.
\end{theoremi}

The following is the equivariant analogue of the Kirchberg-Phillips $\CO_2$-absorption theorem.

\begin{theoremi}[see Theorem \ref{O2-absorbing-uniqueness}] \label{O2-uniqueness-intro}
For every action $\alpha: G\curvearrowright A$ on a unital Kirchberg algebra, the action $\alpha\otimes\delta: G\curvearrowright A\otimes\CO_2$ is strongly cocycle conjugate to $\delta: G\curvearrowright\CO_2$. If $\alpha$ is moreover pointwise outer, then $\alpha\otimes\id_{\CO_2}$ is strongly cocycle conjugate to $\delta$.
\end{theoremi}

In particular, this yields the desired Kirchberg-Phillips-type result motivated above. 
Due to Theorem \ref{model-absorption-intro}, the action $\gamma$ is characterized as the unique action satisfying a university-like condition.
However, it is natural to expect that $\gamma$ should have a kind of natural abstract characterization among all pointwise outer actions on $\CO_\infty$ that can actually be checked on the level of examples. In this vein, we will indeed prove:

\begin{theoremi}[see Corollary \ref{abstract-models-Oinf}] \label{abstract-models-intro}
A pointwise outer action $\beta: G\curvearrowright\CO_\infty$ is strongly cocycle conjugate to $\gamma$ if and only if $\beta$ is approximately representable and the inclusion $\cstar(G)\subset\CO_\infty\rtimes_\beta G$ is a $KK$-equivalence.
\end{theoremi}

In fact, an analogous result is proved with respect to actions on any strongly self-absorbing Kirchberg algebra in place of $\CO_\infty$; see Theorem \ref{approx-rep-ssa-uniqueness}.

It is tempting to conjecture that in fact all $G$-actions on $\CO_\infty$ (or more generally $G$-actions on strongly self-absorbing Kirchberg algebras) are approximately representable, but this problem remains open.
The $K$-theoretic condition on the crossed product is necessary in general (see Remark \ref{Z2-example}), but we will show that it becomes redundant if $G$ has no torsion.
If this is the case, it turns out that for every strongly self-absorbing \cstar-algebra $\CD$ and any $G$-action on $\CD$, the $K$-theoretic data of the crossed product $\CD\rtimes G$ does not depend on the action, and in particular coincides with that of $\CD\otimes\cstar(G)$ naturally.
This follows from a combination of the structure theory of strongly self-absorbing \cstar-algebras and Baum-Connes for amenable groups \cite{MeyerNest06}. 
In this case, the uniqueness result becomes:

\begin{theoremi}[see Theorem \ref{approx-rep-tf-uniqueness}]
Let $G$ be torsion-free and $\CD$ a strongly self-absorbing Kirchberg algebra. Then up to strong cocycle conjugacy, the action $\gamma\otimes\id_\CD$ is the unique approximately representable and pointwise outer action on $\CO_\infty\otimes\CD\cong\CD$.
\end{theoremi}

Finally, we will provide some hard evidence towards the existence of a possible equivariant Kirchberg-Phillips classification theory via equivariant $KK$-theory that applies to all outer actions of amenable group on Kirchberg algebras.
In particular, we prove the following rigid behavior with respect to homotopy equivalence:

\begin{theoremi} \label{homotopy-rigid-intro}
Let $\alpha: G\curvearrowright A$ and $\beta: G\curvearrowright B$ be two pointwise outer actions on unital Kirchberg algebras.
If $(A,\alpha)$ and $(B,\beta)$ are $G$-equivariantly homotopy equivalent, then they are conjugate.
\end{theoremi}

To my knowledge, the combination of Theorems \ref{O2-uniqueness-intro}, \ref{abstract-models-intro} and \ref{homotopy-rigid-intro} marks the first \cstar-dynamical classification result up to cocycle conjugacy that is applicable to actions of all amenable groups.
It is worth mentioning that Sato has earlier proved a classification \cite[Theorem 3.1 and Corollary 3.5]{Sato11} of strongly outer actions on certain monotracial \cstar-algebras, which is also applicable to all amenable groups, and with respect to conjugacy modulo weakly inner automorphisms. 

The technical parts of the proofs of these main results largely involve the notion of countably quantifier-free (semi-)saturated \cstar-dynamical systems over discrete groups, which we introduce in the first section. 
This is a concept inspired by continuous model theory and is a generalization of a countably quantifier-free saturated \cstar-algebra (see \cite{FarahHart13}) to the situation of group actions. 
Intuitively speaking, a saturated action of a discrete group $G$ on a \cstar-algebra shares most of the important properties of an ultraproduct $G$-action, whereas a semi-saturated action shares most of the important properties of a $G$-action induced on a generalized central sequence algebra in Kirchberg's sense \cite{Kirchberg04}. 
The main motivation for this notion comes from the fact that the class of (semi-)saturated \cstar-dynamical systems has reasonable permanence properties, unlike the class of ultraproduct-type actions; these permanence properties later play key roles in the proofs related to the main results.
This leads to overall cleaner, more transparent and conceptual proofs in the main body of the paper.
It would likely be possible to prove the main results of this paper in an ad-hoc way only using ultraproducts and appealing to Kirchberg's $\eps$-test repeatedly, but this would lead to a messier presentation.

In the second section, we will study saturated and pointwise outer actions on unital, simple and purely infinite \cstar-algebras.
Perhaps the most important technical observation for amenable groups in this context is that the fixed point algebras of such actions are again simple and purely infinite, even if the acting group $G$ is infinite.
This allows one to mimic some techniques that would otherwise only be available for finite group actions, such as Connes' $2\times 2$-matrix trick from \cite{Connes77} as it was utilized in \cite{Izumi04}.
This technique will lead us to prove a special cohomology vanishing result for specific kinds of cocycles, which will form a useful technical tool for our purposes.
As an application of this technical tool, we prove an existence result for $G$-equivariant $*$-homomorphisms from certain model actions into certain saturated $G$-\cstar-dynamical systems, in part drawing inspiration from \cite{GoldsteinIzumi11}.
This approach is in line with the philosophy that the concept of saturation (in its various forms), just like ultrapowers, allows one to carry over certain techniques to solve approximation problems that would otherwise be inaccessible outside of a finitary situation.   

The technical results obtained in the second section are then applied in the other sections to induced actions on central sequence algebras in order to deduce the aforementioned main results.
Other important ingredients are techniques from \cite{Szabo16ssa, Szabo16ssa2, Szabo17ssa3} related to strongly self-absorbing actions and $KK$-theoretic computations with the help of Baum-Connes \cite{MeyerNest06} for amenable groups.

In the sixth section, we use a modified argument of Phillips \cite{Phillips97} in order to prove a uniqueness theorem for equivariant $*$-homomorphisms with respect to homotopy, leading to Theorem \ref{homotopy-rigid-intro}.
We then discuss a highly non-trivial application:\ any faithful tensorial Bernoulli shift action of an amenable group $G$ on $\CO_\infty$ satisfies the requirements in Theorem \ref{abstract-models-intro}, and is therefore strongly cocycle conjugate to the model action $\gamma$ defined above; see Corollary \ref{bern-models}.

In the last section, we provide, as a byproduct of our methods, a short proof that pointwise outer actions of amenable, residually finite groups on Kirchberg algebras have Rokhlin dimension at most one in the sense of \cite{HirshbergWinterZacharias15, SzaboWuZacharias15}.\vspace{1em}

\textbf{Acknowledgements.} The work presented in this paper has benefited from visits of the author to the Department of Mathematics at the University of Kyoto and the Mittag-Leffler institute for the program {\em Classification of Operator Algebras:\ Complexity, Rigidity, and Dynamics}. 
I thank both institutes for their kind hospitality and support.
I would moreover like to express my gratitude to Yuki Arano, Sel{\c c}uk Barlak, Masaki Izumi and Aaron Tikuisis for helpful discussions related to this paper.


\section{Saturated \cstar-dynamical systems}

In this section, we will define a notion of countably quantifier-free (semi-)saturation in the context of model theory for \cstar-dynamical systems. This is in analogy to the definition from \cite{FarahHart13} without group actions. We then show some basic properties of (semi-)saturated \cstar-dynamical systems. Our motivation is to apply the results within this abstract setting to certain ultrapowers or central sequence algebras of \cstar-dynamical systems. 

\begin{nota}
Unless specified otherwise, we will stick to the following notation:
\begin{itemize}
\item For a \cstar-algebra $A$, we denote by $A_1$ the closed unit ball.
\item If $F\subset M$ is a finite subset of some set, we write $F\fin M$.
\item If $\alpha: G\curvearrowright A$ is a discrete group action on a \cstar-algebra, we denote by $\lambda^\alpha: G\to\CU(\CM(A\rtimes_\alpha G))$ the canonical unitary representation. Since $G$ will usually be amenable in applications, we will not worry about the distinction between reduced and full crossed products, but we will always mean the full one if it is unspecified.
\item For two discrete group actions $\alpha: G\curvearrowright A$ and $\beta: G\curvearrowright B$ on \cstar-algebras, we write $\alpha\cc\beta$ (resp.\ $\alpha\scc\beta$) to mean that $\alpha$ and $\beta$ are (strongly) cocycle conjugate.
\end{itemize}
\end{nota}

\begin{defi}
Let $D$ be a \cstar-algebra, $G$ a countable, discrete group and $\alpha: G\curvearrowright D$ an action. An $\alpha$-$*$-polynomial $P$ with coefficients in $D$ in the variables $x_1,\dots,x_n$ is the formal expression 
\[
P(x_1,\dots, x_n) = Q\big( \alpha_{g_{1,1}}(x_1),\dots,\alpha_{g_{1,m_1}}(x_1),\dots,\alpha_{g_{n,1}}(x_n),\dots,\alpha_{g_{n,m_n}}(x_n) \big)
\]
for a (noncommutative) $*$-polynomial $Q$ with coefficients in $D$ in $m=m_1+\dots+m_n$ variables and some finite subset $\set{g_{l,j} ~|~ 1\leq l\leq n,~1\leq j\leq m_l }\fin G$. 
\end{defi}

\begin{rem}
Given an $\alpha$-$*$-polynomial $P$, we will naturally identify $P$ as an $\alpha$-$*$-polynomial in countably many variables $\bar{x}=(x_k)_{k\in\IN}$, by extending $P$ trivially as to not depend on the variables $x_k$ with $k>n$. This makes some notation easier.
\end{rem}

\begin{defi} Let $D$ be a \cstar-algebra, $G$ a countable, discrete group and $\alpha: G\curvearrowright D$ an action. Let $\Phi$ be a family of $\alpha$-$*$-polynomials with coefficients in $D$. The action $\alpha$ is called $\Phi$-saturated, if the following holds:
 
For every countable family of $\alpha$-$*$-polynomials $P_n\in\Phi$ and every countable family of compact sets $K_n\subset\IR^{\geq 0}$, 
the following are equivalent:
\begin{enumerate}[label={(\roman*)},leftmargin=*]
\item There is a sequence $b_n\in D_{1}$, such that $\|P_l(\bar{b})\|\in K_l$ for all $l\in\IN$.
\item For every $m\in\IN$ and $\eps>0$, there is a sequence $b_n\in D_1$ such that $\|P_l(\bar{b})\|\in_\eps K_l$ for all $l\leq m$.
\end{enumerate}
\end{defi}

\begin{rem} \label{farah-hart}
Obviously, the above definition recovers ordinary $\Phi$-saturation in the sense of \cite{FarahHart13}, for $\Phi$ consisting of ordinary $*$-polynomials, by just inserting the trivial group as $G$. 
As in the classical situation, an elementary compactness argument shows that it suffices to consider the compact sets $K_n\subset\IR^{\geq 0}$ above to be singletons. If $\Phi$ is closed under scalar multiples, we may further renormalize, so it is even enough to consider the one-point sets $\set{0}$ and $\set{1}$.
Upon dropping even the latter and only considering $\set{0}$, one arrives at what we call semi-saturation.
\end{rem}

\begin{defi}
Let $D$ be a \cstar-algebra, $G$ a countable, discrete group and $\alpha: G\curvearrowright D$ an action. Let $\Phi$ be a family of $\alpha$-$*$-polynomials with coefficients in $D$. The action $\alpha$ is called semi-$\Phi$-saturated, if the following holds:
 
For every countable family of $\alpha$-$*$-polynomials $P_n\in\Phi$, 
the following are equivalent:
\begin{enumerate}[label={(\roman*)},leftmargin=*]
\item There is a sequence $b_n\in D_{1}$, such that $\|P_l(\bar{b})\|=0$ for all $l\in\IN$.
\item For every $m\in\IN$ and $\eps>0$, there is a sequence $b_n\in D_1$ such that $\|P_l(\bar{b})\|\leq\eps$ for all $l\leq m$.
\end{enumerate}
\end{defi}

\begin{nota}
The following special case will be of particular interest. Let $(D,\alpha,G)$ be as above and let $\Phi_{\mathrm{all}}$ be the set of \emph{all} $\alpha$-$*$-polynomials with coefficients in $D$. Then $\alpha$ is called countably quantifier-free (semi-)saturated, if it is (semi-)$\Phi_{\mathrm{all}}$-saturated. For brevity, we will only say (semi-)saturated and omit the term {\it countably quantifier-free} for the rest of this paper.
\end{nota}

\begin{reme}
Potentially, the sentence ``the action $\alpha$ is saturated'', interpreted as above, is in conflict with a different notion called saturation for certain group actions on \cstar-algebras.
In the case that two notions ever actually came into conflict somewhere, it might be appropriate to avoid the abbreviation above and use the long expression {\it countably quantifier-free saturated}. 
Since this conflict of terminology is entirely non-existent within this paper, we will ignore this issue. 
An analogous remark applies for the notion of semi-saturation.
\end{reme}

\begin{example} \label{ex:ultra}
For any sequence $(A^{(n)},\alpha^{(n)})$ of $G$-\cstar-dynamical systems, we get an ultraproduct system $\bigl( \prod_\omega A^{(n)}, \prod_\omega \alpha^{(n)} \bigl)$ via the action 
\[
\prod_\omega \alpha^{(n)}: G\curvearrowright \prod_\omega A^{(n)} = \prod_{n\in\IN} A^{(n)}/\set{(a_n)_n \mid \lim_{n\to\omega} \|a_n\| = 0}.
\]
that is induced componentwise.
A \cstar-dynamical system of this form is saturated. 
\end{example}
\begin{proof}
To see this, one can either appeal to equivariant model theory of metric structures \cite{GardellaLupini16} or apply Kirchberg's $\eps$-test directly, cf.\ \cite[A.4]{Kirchberg04} or \cite[3.1]{KirchbergRordam14}. We shall do the latter. 

With slight abuse of notation, we use the shortcut $\bigl( \prod_\omega A^{(n)}, \prod_\omega \alpha^{(n)} \bigl) = (A,\alpha)$. Let $\pi_\omega: \prod_{n\in\IN} A^{(n)}\to A$ be the quotient map.
Let $P_k$ be a sequence of $\alpha$-$*$-polynomials with coefficients in $A$ and let $\lambda_k\geq 0$ be a sequence of numbers. For each $k\in\IN$, we may choose a sequence $P_k^{(n)}$ of $\alpha^{(n)}$-$*$-polynomials with coefficients in $A^{(n)}$ such that
\[
\pi_\omega\Big( \big( P_k^{(n)}(\bar{x}_n) \big)_n \Big) = P_k\Big( \pi_\omega\big( (\bar{x}_n)_n \big) \Big)\quad\text{for all}~(\bar{x}_n)_n\in \Big( \prod_\IN A^{(n)} \Big)_1^\IN.
\] 
For every $n\in\IN$, let $X_n$ be the set of all contractions in $A^{(n)}$. For every $k, n\in\IN$, define the map
\[
f_k^{(n)}: X_n^\IN\to [0,\infty), \quad f_k^{(n)}(\bar{x}) = \Big| \|P_k^{(n)}(\bar{x})\|-\lambda_k \Big|. 
\]
Also set
\[
f_k: \prod_{n\in\IN} X_n^\IN\to [0,\infty], \quad f_k\big( (\bar{x}_n)_n \big) = \lim_{n\to\omega} f_k^{(n)}(\bar{x}_n). 
\]
Now suppose that for every $\eps>0$ and $m\in\IN$, there is an infinite tuple of contractions $\bar{y}=(y_l)_l\in A^\IN$ such that $\lambda_k-\eps\leq \|P_k(\bar{y})\|\leq \lambda_k+\eps$ for all $k\leq m$. By lifting every contraction $y_l$ to a sequence of contractions $x_{l,n}\in A^{(n)}$, we obtain an element
\[
(x_{l,n})_{l,n\in\IN} = (\bar{x}_n)_n \in \prod_{n\in\IN} X_n^\IN
\]
with
\[
f_k(\bar{x}) = \lim_{n\to\omega} f^{(n)}_k(\bar{x}_n) = \lim_{n\to\omega} \Big| \|P^{(n)}_k(\bar{x_n})\|-\lambda_k \Big| \leq \eps
\]
for every $k\leq m$. As $m$ and $\eps$ were arbitrary, it follows from the $\eps$-test that there exists some element
\[
\bar{z}=(\bar{z}_n) = (z_{l,n})_{l,n\in\IN} \in \prod_{n\in\IN} X_n^\IN
\]
with
\[
0 = f_k(\bar{z}) = \lim_{n\to\omega} f_k^{(n)}(\bar{z}_n) = \lim_{n\to\omega} \Big| \|P_k^{(n)}(\bar{z}_n)\|-\lambda_k \Big|
\]
for every $k\in\IN$. But this means that 
\[
\bar{s}=\pi_\omega\big( (\bar{z}_n)_n \big) \in A_1^\IN
\]
satisfies
\[
\|P_k(\bar{s})\| = \lambda_k\quad\text{for all}~k\in\IN.
\]
Since the $P_k$ and $\lambda_k$ were arbitrary, this shows our claim.
\end{proof}

For the rest of this section, we will establish elementary properties of saturated and semi-saturated \cstar-dynamical systems. The proofs are somewhat tedious and elementary, but necessary for a slick presentation in the main body of the paper. Upon the first few readings, the reader might thus prefer to skip most of the proofs in this section.

Our first observation is that semi-saturation is equivalent to saturation as long as the underlying \cstar-algebra is simple and purely infinite. This follows from the following elementary observations:

\begin{prop}
\label{norms-polynomials}
Let $A$ be a \cstar-algebra and $x\in A$ an element. Let $\lambda>0$.
\begin{enumerate}[label=\textup{(\arabic*)},leftmargin=*]
\item $\|x\|\leq\lambda$ if and only if there exists $y\in A$ such that $\frac{x^*x}{\lambda^2}+y^*y-y-y^*=0$. If this is the case, then $y$ can be chosen to be a positive contraction.\label{norms-polynomials:1}
\item Assume that $A$ is simple and purely infinite and $e\in A$ a positive element of norm one. Then $\|x\|\geq\lambda$ if and only if there is $a\in A, \|a\|\leq\lambda^{-1}$ with $a^*x^*xa=e$. \label{norms-polynomials:2}
\end{enumerate}
\end{prop}
\begin{proof}
\ref{norms-polynomials:1}: ``$\Longrightarrow$'' follows from setting $y=\eins-\sqrt{\eins-\frac{x^*x}{\lambda^2}}$.

``$\Longleftarrow$'' follows from computing
\[
0=\frac{x^*x}{\lambda^2}+y^*y-y-y^* \iff \eins-\frac{x^*x}{\lambda^2}=\eins+y^*y-y-y^* = (\eins-y)^*(\eins-y)
\]
In particular, $\eins-\frac{x^*x}{\lambda^2}\geq 0$, which is equivalent to $\|x\|\leq\lambda$.

\ref{norms-polynomials:2}: The implication ``$\Longleftarrow$'' is trivial in every \cstar-algebra, and ``$\Longrightarrow$'' follows directly from the assumption that $A$ is simple and purely infinite.
\end{proof}

\begin{cor}
\label{spi-sat}
Let $D$ be a simple and purely infinite \cstar-algebra. Let $\alpha: G\curvearrowright D$ be an action of a countable, discrete group. Then $\alpha$ is saturated if and only if $\alpha$ is semi-saturated.
\end{cor}
\begin{proof}
Assume that $D$ is simple and purely infinite, and that $\alpha$ is semi-saturated. Let us show that $\alpha$ is saturated. Let $\set{P_k}_{k\in\IN}$ be a countable collection of $\alpha$-$*$-polynomials, and let $\lambda_k\in\set{0,1}$ be a sequence. Assume that for every $\eps>0$ and $m\in\IN$, there exists a tuple $\bar{b}\in D_1^\IN$ with 
\[
\lambda_k-\eps\leq \|P_k(\bar{b})\| \leq \eps+\lambda_k \quad\text{for all}~k\leq m.
\]
By Remark \ref{farah-hart}, it suffices to show that we can find a tuple $\bar{b}\in D_1^\IN$ with $\|P_k(\bar{b})\|=\lambda_k$ for all $k$. Choose some element $e\in D$ of norm one. Let us consider the countable collection of $\alpha$-$*$-polynomials $Q_k^{(i)}$ for $k\in\IN$ and $i=1,2$ given by
\[
Q_k^{(i)}(\bar{b},\bar{x},\bar{y}) = \begin{cases} P_k(\bar{b}) &,\quad \lambda_k=0 \\
P_k(\bar{b})^*P_k(\bar{b})+x_k^*x_k-x_k-x_k^* &,\quad \lambda_k=1,\ i=1 \\
y_k^*P_k(\bar{b})^*P_k(\bar{b})y_k-e &,\quad \lambda_k=1,\ i=2.
\end{cases}
\]
Then by Proposition \ref{norms-polynomials}, we see that for a given $k\in\IN$ with $\lambda_k=1$ and a tuple $\bar{b}\in D_1^\IN$, the norm $\|P_k(\bar{b})\|$ is close to $1$ if and only if there exist tuples $\bar{x},\bar{y}\in D_1^\IN$ such that $\|Q_k^{(1)}(\bar{b},\bar{x},\bar{y})\|$ and $\|Q_k^{(2)}(\bar{b},\bar{x},\bar{y})\|$ are close to zero. 
In particular, for every $m\in\IN$, the norm $\|P_k(\bar{b})\|$ is close to $\lambda_k$ for all $k\leq m$ if and only if there exist tuples $\bar{x},\bar{y}\in D_1^\IN$ such that $\|Q_k^{(1)}(\bar{b},\bar{x},\bar{y})\|$ and $\|Q_k^{(2)}(\bar{b},\bar{x},\bar{y})\|$ are close to zero for all $k\leq m$. 
By our initial assumption, it thus follows that for every $m\in\IN$ and $\eps>0$, we find tuples $\bar{b},\bar{x},\bar{y}\in D_1^\IN$ with 
\[
\|Q_k^{(1)}(\bar{b},\bar{x},\bar{y})\| + \|Q_k^{(2)}(\bar{b},\bar{x},\bar{y})\|\leq\eps\quad\text{for all}~k\leq m.
\]
As $\alpha$ is semi-saturated, it follows that we find such $\bar{b},\bar{x},\bar{y}$ with $\|Q_k^{(i)}(\bar{b},\bar{x},\bar{y})\|=0$ for all $i=1,2$ and $k\in\IN$. By Proposition \ref{norms-polynomials}, this implies $\|P_k(\bar{b})\|=\lambda_k$ for all $k\in\IN$, which shows our claim.
\end{proof}

As we are only going to study actions on purely infinite \cstar-algebras in this paper, Corollary \ref{spi-sat} leads us to mainly restrict our attention to semi-saturated actions for the rest of this section.

The following is a trivial, but useful observation:

\begin{prop} \label{polynomial subalgebra}
Let $(D,\alpha)$ be a (semi-)saturated \cstar-dynamical system. Let $A\subset D$ be some $\alpha$-invariant \cstar-subalgebra. If there exists a collection $\set{P_n}_{n\in\IN}$ of $\alpha$-$*$-polynomials in one variable with coefficients in $D$ such that
\[
A=\set{ d\in D ~|~ P_n(d)=0~\text{for all}~n },
\]
then the restricted \cstar-dynamical system $(A,\alpha|_A)$ is also (semi-)saturated.
\end{prop}

As a consequence, we can immediately record some permanence properties of semi-saturated \cstar-dynamical systems. We note that even though \ref{semi-sat-permanence:6} below seems obvious at first sight, the detailed proof requires some work.
 
\begin{prop} \label{semi-sat-permanence}
Let $(D,\alpha)$ be a semi-saturated $G$-\cstar-dynamical system.
\begin{enumerate}[label=\textup{(\roman*)},leftmargin=*]
\item If $p\in D$ is an $\alpha$-invariant projection, then the induced system on the corner $(pDp, \alpha|_{pDp})$ is semi-saturated. \label{semi-sat-permanence:1}
\item The fixed point algebra $D^\alpha$ is a semi-saturated \cstar-algebra. \label{semi-sat-permanence:2}
\item If $\beta: G\curvearrowright D$ is another action such that $\alpha_g$ and $\beta_g$ are unitarily equivalent for all $g\in G$, then the system $(D, \beta)$ is semi-saturated. \label{semi-sat-permanence:3}
\item For every group $H$ and group homomorphism $\pi: H\to G$, the system $(D,\alpha\circ\pi,H)$ is semi-saturated. If $\pi$ is surjective, then the converse holds. \label{semi-sat-permanence:4}
\item For any separable and $\alpha$-invariant \cstar-subalgebra $A\subset D$, the system $(D\cap A', \alpha|_{D\cap A'})$ is semi-saturated. \label{semi-sat-permanence:5}
\item For any $n\in\IN$, the system $(M_n\otimes D, \id_{M_n}\otimes\alpha)$ is semi-saturated. \label{semi-sat-permanence:6}
\end{enumerate}
\end{prop}
\begin{proof}
\ref{semi-sat-permanence:1}, \ref{semi-sat-permanence:2}, \ref{semi-sat-permanence:5}: We have
\[
pDp = \set{ x\in D \mid xp=x=px}
\] 
and
\[
D^\alpha = \set{ x\in D \mid x=\alpha_g(x) ~\text{for all}~g\in G}.
\]
If in \ref{semi-sat-permanence:5}, $a_k\in A$ is some dense sequence, then moreover
\[
D\cap A' = \set{ x\in D \mid xa_k=a_kx ~\text{for all}~ k\in\IN }.
\]
Thus, these are special cases of Proposition \ref{polynomial subalgebra}.

\ref{semi-sat-permanence:3}, \ref{semi-sat-permanence:4}: This is obvious, since in this case, the set of equivariant $*$-polynomials either remains the same or becomes smaller.

\ref{semi-sat-permanence:6}: Denote $D^{(n)}=M_n(D)=M_n\otimes D$ and $\alpha^{(n)}=\id_{M_n}\otimes\alpha$.
Every tuple $\bar{x}\in M_n(D)^\IN$ can be written as $(\bar{x}_{i,j})_{1\leq i,j\leq n}$ for tuples $\bar{x}_{i,j}\in D^\IN$. Accordingly, define the bijection
\[
c: M_n(D)^\IN \to D^{n^2\times\IN},\quad c(\bar{x}) = (\bar{x}_{i,j})_{i,j}
\]
in the obvious way. Observe that the map $c$ sends every tuple of contractions to a tuple of contractions.

Let $P$ be an $\alpha^{(n)}$-$*$-polynomial with coefficients in $D^{(n)}$. Since the multiplication on $M_n(D)$ is given by matrix multiplication, it is not hard to see that one has
\[
P(\bar{x}) = \Big( P^{(i,j)}(c(\bar{x})) \Big)_{1\leq i,j\leq n}
\]
for some $\alpha$-$*$-polynomials $P^{(i,j)}$, $1\leq i,j\leq n$, with coefficients in $D$.

Now let $Q_k$ be a sequence of $\alpha^{(n)}$-$*$-polynomials with coefficients in $D^{(n)}$ and variables $\bar{x}$. Let $\bar{y}$ be another countable tuple free variables, and set $\bar{z}=(\bar{x},\bar{y})$. We let
\[
R_k(\bar{z})=x_k^*x_k+y_k^*y_k-y_k-y_k^*,\quad k\in\IN.
\]
For notational convenience, let us consider the countable collection of $\alpha^{(n)}$-$*$-polynomials given by
\[
P_{2k}(\bar{z})=Q_k(\bar{x}),\quad P_{2k-1}(\bar{z})=R_k(\bar{z}),\quad k\in\IN.
\]
For every $k$, choose a decomposition $P_k^{(i,j)}$ for $P_k$ as above. 

Assume that there exists a sequence $\bar{x}^{(l)}\in M_n(D)_1^\IN$ such that
\[
\|Q_k(\bar{x}^{(l)})\|\stackrel{l\to\infty}{\longrightarrow} 0 \quad\text{for every}~ k.
\]
As $\bar{x}^{(l)}$ consists of contractions, there exists (by Proposition \ref{norms-polynomials}) a tuple $\bar{y}^{(l)}\in M_n(D)_1^\IN$ such that $R_k(\bar{x}^{(l)},\bar{y}^{(l)})=0$ for all $k$. In particular, we obtain a sequence $\bar{z}^{(l)}\in M_n(D)_1^\IN$ such that $\|P_k(\bar{z}^{(l)})\|\stackrel{l\to\infty}{\longrightarrow} 0$ for every $k$.

Then the tuples $c(\bar{z}^{(l)})\in D_1^{n^2\times\IN}$ satisfy
\[
\|P^{(i,j)}_k(c(\bar{z}^{(l)}))\|\leq \|P_k(\bar{z}^{(l)})\|\stackrel{l\to\infty}{\longrightarrow} 0
\]
for all $k$ and $1\leq i,j\leq n$. Since $(D,\alpha)$ is semi-saturated, this implies that there is some tuple $\bar{b}\in D_1^{n^2\times\IN}$ with $P^{(i,j)}_k(\bar{b})=0$ for all $k$ and $0\leq i,j\leq n$. Setting $\bar{z}=c^{-1}(\bar{b})\in M_n(D)^\IN$ yields a tuple satisfying
\[
P_k(\bar{z}) = \Big( P^{(i,j)}(c(\bar{z})) \Big)_{1\leq i,j\leq n} = 0
\]
for all $k$. From the way we defined the polynomials $P_k$ on $\bar{z}=(\bar{x},\bar{y})$, it follows from Proposition \ref{norms-polynomials} that 
\[
Q_k(\bar{x})=0\quad\text{and}\quad \|x_k\|\leq 1
\]
for all $k$.
Since the $Q_k$ were arbitrary, this finishes the proof.
\end{proof}

In addition to the relative commutant, it can often be more useful to consider a generalized relative commutant, which is analogous to how Kirchberg has defined a corrected version of the central sequence algebra in \cite{Kirchberg04}.

\begin{defi}[cf.~{\cite[1.1]{Kirchberg04}}] \label{Kirchberg nota}
Let $D$ be a semi-saturated \cstar-algebra and $C\subset D$ some \cstar-subalgebra. Denote the two-sided annihilator of $C$ inside $D$ by 
\[
\ann(C,D)=\set{x\in D ~|~ ax=xa=0~\text{for all}~a\in C}.
\] 
Then $\ann(C,D)$ is easily seen to be a closed and two-sided ideal in the relative commutant $D\cap C'$.
Define the generalized relative commutant of $C$ in $D$ as
\[
F(C,D) = D\cap C'/\ann(C,D).
\]
We will write $\pi_C: D\cap C'\to F(C,D)$ for the quotient map.
\end{defi}

\begin{rem} \label{product map}
Let us make the following few observations about generalized relative commutants.
\begin{enumerate}[label={(\roman*)},leftmargin=*]
\item If $C$ is $\sigma$-unital, then semi-saturation of $D$ enables us to find a positive element $e\in D$ acting as a unit on $C$. Then the element $e+\ann(C,D)\in F(C,D)$ defines a unit. This means that $F(C,D)$ is unital, whenever $C\neq 0$ is $\sigma$-unital. \label{product map:1}
\item One has a canonical $*$-homomorphism 
\[
\Phi_C: F(C,D)\otimes_{\max} C\to \overline{CDC}\subset D
\]
given by 
\[
\big( x+\ann(C,D) \big)\otimes a \mapsto x\cdot a.
\]
Assuming that $C$ is $\sigma$-unital, the element $\eins\otimes a$ is sent to $a$ for every $a\in C$ under this $*$-homomorphism. Moreover, the norm of an element $y\in F(C,D)$ is given by 
\[
\|y\|=\sup\set{ \|\Phi_C(y\otimes a)\| \mid 0\leq a\in A,\ \|a\|\leq 1 }. 
\] 
\label{product map:2}
\item Given an automorphism $\phi\in\Aut(D)$ with $\phi(C)=C$, one has 
\[
\phi(D\cap C') = D\cap C'\quad\text{and}\quad \phi(\ann(C,D))=\ann(C,D). 
\]
In particular, we obtain an automorphism $\tilde{\phi}$ on $F(C,D)$ in a natural way. If more generally $\alpha: G\curvearrowright D$ is an action such that $C$ is $\alpha$-invariant, then we have an induced action $\tilde{\alpha}: G\curvearrowright F(C,D)$. The $*$-homomorphism $\Phi_C$ from above is then equivariant with regard to $\tilde{\alpha}\otimes\alpha|_C$ and $\alpha$. \label{product map:3}
\end{enumerate}
\end{rem}
  
It seems to be unclear whether the following assertion holds for saturation instead of semi-saturation.

\begin{prop} \label{F(A)-saturated}
Let $\alpha: G\curvearrowright D$ be a semi-saturated action. For any separable and $\alpha$-invariant \cstar-subalgebra $A\subset D$, the system $( F(A,D), \tilde{\alpha} )$ is semi-saturated.
\end{prop}
\begin{proof}
Let $P_k$ be a sequence of $\tilde{\alpha}$-$*$-polynomials with coefficients in $F(A,D)$. Lift each $P_k$ to an $\alpha$-$*$-polynomial $Q_k$ with coefficients in $D\cap A' \subseteq D$. Let $\bar{y}^{(n)}\in F(A,D)_1^\IN$ be a sequence with $\|P_k(\bar{y}^{(n)})\|\to 0$ for every $k$. Pick $\bar{x}^{(n)}\in (D\cap A')_1^\IN$ with $\pi_A^\IN(\bar{x}^{(n)})=\bar{y}^{(n)}$. We have for every $k$ and $a\in A$ that
\[
\|Q_k(\bar{x}^{(n)})\cdot a\| = \|\Phi_A\big( P_k(\bar{y}^{(n)})\otimes a\big)\|\stackrel{n\to\infty}{\longrightarrow} 0. 
\]
Let $S\subset A_1$ be a countable, dense set in the unit ball. Consider the countable family of $\alpha$-$*$-polynomials with coefficients in $D$ given by
\[
H^{(1)}_{l,a}(\bar{x}) = [x_l, a],\quad l\in\IN,~a\in S
\]
and
\[
H^{(2)}_{k,a}(\bar{x}) = Q_k(\bar{x})\cdot a,\quad k\in\IN,~a\in S.
\]
We then have
\[
H^{(1)}_{l,a}(\bar{x}^{(n)})=0,\quad \|H^{(2)}_{k,a}(\bar{x}^{(n)})\|\stackrel{n\to\infty}{\longrightarrow} 0
\]
for all $k,l\in\IN$ and $a\in S$. As $\alpha$ is semi-saturated, there exists a tuple $\bar{x}\in D_1^\IN$ with
\[
H^{(1)}_{l,a}(\bar{x}^{(n)})=0=H^{(2)}_{k,a}(\bar{x}^{(n)})
\]
for all $k,l\in\IN$ and $a\in S$. The first equation yields $\bar{x}\in (D\cap A')_1^\IN$ and the second equation yields
\[
\|P_k(\pi_A^\IN(\bar{x}))\| = \sup_{a\in S} \|Q_k(\bar{x})\cdot a\| = 0
\]
for all $k$. As the $P_k$ were arbitrary, this shows our claim.
\end{proof}

Let us now record a very useful characterization of simple and purely infinite \cstar-algebras, which is due to Kirchberg \cite[Chapter 3, Section 1]{KirchbergC}.

\begin{theorem} \label{1-step ai}
Let $D$ be a non-zero \cstar-algebra. Then $D$ is simple and purely infinite if and only if the following holds:\
For every separable \cstar-subalgebra $C\subseteq D$ and every nuclear c.p.c.~map $V: C\to D$, there are two sequences of contractions $d_n, e_n\in D$ with $V(c)=\lim_{n\to\infty} d_n^*cd_n = \lim_{n\to\infty} e_n^*ce_n$ for all $c\in C$ and $0=\lim_{n\to\infty} e_n^*d_n$.
\end{theorem}

The following is a modification of a well-known theorem of Kirchberg-Phillips \cite{KirchbergPhillips00}, which has fundamental importance for the classification of Kirchberg algebras; see also \cite[Chapter 7, Section 1]{Rordam}.

\begin{theorem}[cf.~{\cite[Lemma 3.3 and Proposition 3.4]{KirchbergPhillips00}}] \label{F() pis}
Let $D$ be a semi-saturated, simple and purely infinite \cstar-algebra. Let $A\subset D$ be a separable, simple and nuclear \cstar-subalgebra. Then $F(A,D)$ is simple and purely infinite.
\end{theorem}
\begin{proof}
In order to show our claim, it suffices to show the following statement: For every pair of positive contractions $0\leq a,b\in F(A,D)$ with $\sp(b)\subset\sp(a)$, there exists a non-unitary isometry $s\in F(A,D)$ with $b=s^*as$. Clearly, the case $b=\eins$ is enough to deduce that $F(A,D)$ is simple and purely infinite.

So let $a, b$ be given as above. Using functional calculus, we identify $\mathrm{C}^*(a,\eins)$ with $\CC(\sp(a))$ and $\mathrm{C}^*(b,\eins)$ with $\CC(\sp(b))$.
Then by Remark \ref{product map}, we have two $*$-homomorphisms
\[
\phi_a: \CC(\sp(a))\otimes A\to D, \quad f\otimes x\mapsto \Phi_A(f(a)\otimes x)
\]
and
\[
\phi_b: \CC(\sp(b))\otimes A\to D, \quad f\otimes x\mapsto \Phi_A(f(b)\otimes x).
\]
Then both of these $*$-homomorphisms are injective. For example, if $\phi_a$ was not, then by simplicity of $A$, the kernel would be isomorphic to $\CC_0(U)\otimes A$ for some non-empty open subset $U\subset\sp(a)$. But this would imply for some $0\neq f\in\CC(\sp(a))$ that $\Phi_A(f(a)\otimes x)=0$ for every $x\in A$, prompting $f(a)=0$. But this is impossible. The same argument is valid for $\phi_b$.

Let $C\subset D$ be the image of $\phi_a$. Since we have a natural surjection $\pi: \CC(\sp(a))\otimes A\to\CC(\sp(b))\otimes A$ given by restricting functions, we can find a $*$-homomorphism $\psi: C\to D$ with $\psi\circ\phi_a=\phi_b\circ\pi$. 
Since $C\cong\CC(\sp(a))\otimes A$ is nuclear, $\psi$ is a nuclear c.p.c.~map. Since $C$ is also separable, we can apply Theorem \ref{1-step ai} to find two sequences of contractions $d_n,e_n\in D$ with $\psi(c)=\lim_{n\to\infty} d_n^*cd_n=\lim_{n\to\infty} e_n^*ce_n$ for all $c\in C$ and such that $0=\lim_{n\to\infty} e_n^*d_n$.
By semi-saturation of $D$ and separability of $C$, we can then find two contractions $d,e\in D$ with $\psi(c)=d^*cd=e^*ce$ for all $c\in C$ and $e^*d=0$.\footnote{Given a countable, dense set $S\subset C$, finding such elements amounts to solving countably many $*$-polynomial equations given by inserting only $c\in S$.}

For all $x\in A$, the equation 
\[
x=\phi_b(\eins_{\CC(\sp(b))}\otimes x)=\phi_b\circ\pi(\eins_{\CC(\sp(a))}\otimes x) = \psi\circ\phi_a(\eins_{\CC(\sp(a))}\otimes x) = \psi(x)
\]
yields the identity $d^*xd=x=e^*xe$. If $x\in A$ is positive, this implies
\[
(dx-xd)^*(dx-xd) = xd^*dx-d^*xdx-xd^*xd+d^*x^2d = xd^*dx-x^2 \leq 0.
\]
In particular, $dx=xd$ for all $x\in A$. With the same calculation for $e$, we thus get $d,e\in D\cap A'$. Since then both $d^*d$ and $e^*e$ act as units on $A$, their images $s=d+\ann(A,D)$ and $t=e+\ann(A,D)$ are isometries in $F(A,D)$. Moreover, the equation $e^*d=0$ implies $t^*s=0$. In particular, $s$ is a proper isometry in $F(A,D)$.

Finally, let $a=a'+\ann(A,D)$ for some $a'\in D\cap A'$. Then we have for all $x\in A$ that 
\[
\begin{array}{ccccl}
\Phi_A(b\otimes x) &=& \phi_b(\id_{\sp(b)}\otimes x) &=& \phi_b\circ\pi(\id_{\sp(a)}\otimes x) \\
&=& \psi\circ\phi_a(\id_{\sp(a)}\otimes x) &=& \psi\circ\Phi_A(a\otimes x) \\
&=& d^*\Phi_A(a\otimes x)d &=& d^*a'xd \\
&=& d^*a'dx  &=& \Phi_A((s^*as)\otimes x).
\end{array}
\] 
This implies $b=s^*as$ by Remark \ref{product map}\ref{product map:2} and finishes the proof. 
\end{proof}

Lastly, recall the following notion, which we will use throughout:

\begin{defi}
Let $G$ be a countable and discrete group. Let $B$ and $C$ be two unital \cstar-algebras with actions $\beta: G\curvearrowright B$ and $\gamma: G\curvearrowright C$. Let $\phi_1,\phi_2: (B,\beta)\to (C,\gamma)$ be unital and equivariant $*$-homomorphisms. 
\begin{enumerate}[label={(\roman*)}, leftmargin=*]
\item $\phi_1$ and $\phi_2$ are called $G$-unitarily equivalent, if there exists a unitary $v\in\CU(C^\gamma)$ with $\phi_2=\ad(v)\circ\phi_1$.  
\item $\phi_1$ and $\phi_2$ are called approximately $G$-unitarily equivalent, if the following holds:\
For all $F\fin B, M\fin G$ and $\eps>0$, there is a unitary $v\in\CU(C)$ with $\|\phi_2(x)-\ad(v)\circ\phi_1(x)\|\leq\eps$ for all $x\in F$ and $\|\gamma_g(v)-v\|\leq\eps$ for all $g\in M$. 
\end{enumerate}
\end{defi}


\section{Saturated actions on simple and purely infinite \cstar-algebras}

In what follows, we will mainly be interested in those (semi-)saturated \cstar-dynamical systems where the underlying \cstar-algebra is simple and purely infinite, and the action is pointwise outer. 

The earlier part of this section consists of observations that are very reminiscent of \cite[Section 3]{IzumiMatui10}. The following is a straightforward variation of \cite[Lemma 3]{Nakamura99}:

\begin{lemma} \label{Kishimoto Lemma}
Let $D$ be a simple and purely infinite \cstar-algebra. Let $\phi\in\Aut(D)$ be an automorphism such that the induced $\IZ$-action on $D$ is saturated. Assume that $\phi$ is outer. Then for every positive element $a\in D$ of norm one, there exists a non-zero projection $q\leq a$ with $q\phi(q)=0$.
\end{lemma}
\begin{proof}
We apply \cite[Lemma 1.1]{Kishimoto81}; this means that given any $\eps>0$, we may find a positive norm one element $b\in\overline{aDa}$ with $\|b\phi(b)\|\leq\eps$.
Since $D$ is simple and purely infinite, it has real rank zero.
As $b$ has norm one, it follows that we can find a sequence of non-zero projections $q_n\in D$ with $\|q_nb-q_n\|\to 0$. 
By saturation of $D$, we can find a non-zero projection $q\in D$ with $bq=qb=q$.
In particular, this prompts $q\in\overline{aDa}$ and thus $q\leq a$, as well as $\|q\phi(q)\|\leq\eps$.
Because of saturation and as $\eps>0$ was arbitrary, our claim follows.
\end{proof}

\begin{prop} \label{Nakamura Lemma}
Let $G$ be a countable and discrete group. Let $D$ be a simple and purely infinite \cstar-algebra and let $\alpha: G\curvearrowright D$ be a pointwise outer and saturated action. Then for all positive elements $a\in D$ of norm one, there is a non-zero projection $q\leq a$ with $q\alpha_g(q)=0$ for all $g\in G\setminus\set{1_G}$.
\end{prop}
\begin{proof}
For all $g\in G$, the system $(D,\alpha_g,\IZ)$ is also saturated by Corollary \ref{spi-sat} and Proposition \ref{semi-sat-permanence}\ref{semi-sat-permanence:4}.
Let $\IN\to G\setminus\set{1_G}, n\mapsto g_n$ be a surjection.
Apply Lemma \ref{Kishimoto Lemma} and find a non-zero projection $q_1\leq a$ with $q_1\alpha_{g_1}(q_1)=0$.
Inductively, once we have found a non-zero projection $q_n\leq a$ with $q_n\alpha_{g_j}(q_n)=0$ for all $j=1,\dots,n$, we can apply Lemma \ref{Kishimoto Lemma} again to find a non-zero projection $q_{n+1}\leq q_n$ with $q_{n+1}\alpha_{g_{n+1}}(q_{n+1})=0$.
Then this projection automatically satisfies $q_{n+1}\leq a$ and $q_{n+1}\alpha_{g_j}(q_{n+1})=0$ for all $j=1,\dots,n+1$.
Using that $(D,\alpha)$ is saturated, we can find a non-zero projection $q\leq a$ with $q\alpha_g(q)=0$ for all $g\in G\setminus\set{1_g}$, as desired.
\end{proof}

The following lemma is the quintessential ingredient enabling the proofs of virtually all the main results in this paper.

\begin{lemma} \label{fixed point}
Let $G$ be a discrete, countable and amenable group. Let $D$ be a unital, simple and purely infinite \cstar-algebra and let $\alpha: G\curvearrowright D$ be a saturated and pointwise outer action.
Then the fixed point algebra $D^\alpha$ is simple and purely infinite.
\end{lemma}
\begin{proof}
Since this is well-known when $G$ is finite (and under much milder conditions), we may assume that $G$ is countably infinite. 
So let $a\in D^\alpha$ be a positive contraction of norm one.

Applying Proposition \ref{Nakamura Lemma}, we can find a non-zero projection $q\leq a$ in $D$ with $q\alpha_g(q)=0$ for all $g\in G\setminus\set{1_G}$. 
Notice that we have $qa=aq=q$. 
As $a$ is fixed by $\alpha$, we also have $\alpha_g(q)a=\alpha_g(q)$ for all $g\in G$. 
Since $D$ is simple and purely infinite, we can find an isometry $t\in D$ with $\eins=t^*qt$. 
Let $F\fin G$ and $\eps>0$ be given, and choose an $(F,\eps)$-invariant F{\o}lner set $K\fin G$. Define 
\[
s=|K|^{-1/2}\sum_{g\in K} \alpha_g(qt).
\] 
Then $s$ is an isometry because
\[
s^*s=|K|^{-1}\cdot\sum_{g,h\in K} \alpha_g(t^*)\underbrace{ \alpha_g(q)\alpha_h(q)}_{=0~\text{for}~g\neq h} \alpha_h(t) = |K|^{-1}\cdot \sum_{g\in K} \alpha_g(t^*qt) = \eins.
\]
Note that $s$ is not a unitary, as for any $g_0\in G\setminus K$, we have
\[
s = \Big( \sum_{g\in K} \alpha_g(q) \Big) s \quad\text{and}\quad \sum_{g\in K} \alpha_g(q) \leq \eins-\alpha_{g_0}(q).
\]
Let us now show that $s$ is $(F,\eps)$-approximately fixed by $\alpha$.
For an element $g\in F$, denote 
\[
R_g=|K|^{1/2}\cdot \big( s-\alpha_g(s) \big)=\sum_{h_1\in K\setminus gK} \alpha_{h_1}(qt)- \sum_{h_2\in gK\setminus K} \alpha_{h_2}(qt).
\]
Using again that $q$ translates orthogonally with $\alpha$, we obtain
\[
R_g^*R_g = \sum_{h_1\in K\setminus gK} \alpha_{h_1}(t^*qt)+\sum_{h_2\in gK\setminus K} \alpha_{h_2}(t^*qt) = |K\Delta gK|\cdot\eins.
\]
This implies
\[
\begin{array}{ccl}
\|s-\alpha_g(s)\|^2  &=& \|(s-\alpha_g(s))^*(s-\alpha_g(s)\| \\
&=& \dst |K|^{-1}\cdot\|R_g^*R_g\| \\
&=& \dst \frac{|K\Delta gK|}{|K|} ~\leq~ \eps.
\end{array}
\]
Moreover, we also clearly have $as=s$.
Since $F\fin G$ and $\eps>0$ were arbitrary, we can use that $\alpha$ is saturated and find a proper isometry $s\in D^\alpha$ satisfying $as=s$, which in particular implies $s^*as=\eins$.
This finishes the proof. 
\end{proof}

\begin{reme}
It should be pointed out that the proof of Lemma \ref{fixed point} and some of its subsequent applicatione below bear some resemblance to the approach behind \cite[Proposition 4.5]{MatuiSato14} in Matui-Sato's work.
The proper isometry $s$ constructed above can be used to generate a suitable copy of the Cuntz algebra $\CO_\infty$ in a central sequence algebra, leading to equivariant absorption of the trivial action on $\CO_\infty$.
In the proof of \cite[Proposition 4.5]{MatuiSato14}, a sequence of elements $s_n$ in a \cstar-algebra is constructed using techniques related to property (SI) in place of pure infiniteness, but using a very similar idea.
This sequence gives rise to an element in the central sequence algebra, which in a sense generates a copy of the Jiang-Su algebra $\CZ$, leading to \cite[Theorem 4.9]{MatuiSato14}.
For all intents and purposes, Lemma 2.3 might therefore be regarded as a purely infinite analogue of \cite[Proposition 4.5]{MatuiSato14}.
\end{reme}

The following is a fairly low-powered cohomology vanishing result for special cocycles, but has the advantage to be applicable for actions of all amenable groups.
Moreover, it is strong enough to be useful for certain model actions that we will consider in our main results.
It is obtained by adapting a known $2\times 2$ matrix trick originating in work of Connes \cite{Connes77}, which is made possible due to Lemma \ref{fixed point}.

\begin{lemma} \label{O2 coboundary}
Let $G$ be a discrete, countable and amenable group. Let $D$ be a unital, simple and purely infinite \cstar-algebra and let $\alpha: G\curvearrowright D$ be a saturated and pointwise outer action.
Let $u: G\to\CU(\CO_2)$ be a unitary representation. 
Let $\nu: \CO_2\to D^\alpha$ be a unital $*$-homomorphism. 
Then the collection $\set{\nu(u_g)}_{g\in G}$, viewed as an $\alpha$-cocycle in $D$, is a coboundary. 
\end{lemma}
\begin{proof}
Denote $w_g=\nu(u_g)$ for all $g\in G$.
Consider the action $\beta: G\curvearrowright M_2\otimes D$ given by $\beta_g = \ad\matrix{\eins & 0 \\ 0 & w_g}\circ(\id_{M_2}\otimes\alpha_g)$.
Then $\beta$ is saturated and pointwise outer by Proposition \ref{semi-sat-permanence}\ref{semi-sat-permanence:3}+\ref{semi-sat-permanence:6} and Corollary \ref{spi-sat}.
Hence we can apply Lemma \ref{fixed point} to deduce that the fixed point algebra $(M_2\otimes D)^\beta$ is simple and purely infinite. 

Now $\CO_2$ admits approximately central embeddings into itself, i.e., we can find a sequence of unital $*$-homomorphisms $\mu_n: \CO_2\to\CO_2$ such that $\|[x,\mu_n(y)]\|\stackrel{n\to\infty}{\longrightarrow} 0$ for all $x,y\in\CO_2$. 
We have
\[
\big\| [\nu(x), (\nu\circ\mu_n)(y)] \big\| \stackrel{n\to\infty}{\longrightarrow} 0 \quad\text{for all}~ x,y\in\CO_2.
\] 
Applying that $D^\alpha$ is semi-saturated, there exists a unital $*$-homomorphism $\mu: \CO_2\to D^\alpha\cap\nu(\CO_2)'$.\footnote{Note that for this purpose, it is only necessary to find a solution to finitely many polynomial equations in four variables, namely the ones analogous to the universal relations determining $\CO_2\otimes\CO_2$.}
The $*$-homomorphism $\theta: \CO_2\to M_2\otimes D$ given by $\theta(x) = \diag(\mu(x),\mu(x))$ has image in the \cstar-algebra 
\[
(M_2\otimes D)^\beta\cap\set{e_{1,1}\otimes\eins,e_{2,2}\otimes\eins}'.
\] 
In particular, it follows that $[e_{1,1}\otimes\eins]=[e_{2,2}\otimes\eins]=0$ in $K_0((M_2\otimes D)^\beta)$. Since $(M_2\otimes D)^\beta$ is simple and purely infinite, we can find (cf.\ \cite{Cuntz81}) a partial isometry $r\in (M_2\otimes D)^\beta$ with $r^*r=e_{1,1}\otimes\eins$ and $rr^*=e_{2,2}\otimes\eins$. But then $r$ must necessarily be of the form $r=v\otimes e_{2,1}$ for a unitary $v\in D$ satisfying
\[
\matrix{ 0 & 0 \\ v & 0} = \matrix{\eins & 0 \\ 0 & w_g} \matrix{ 0 & 0 \\ \alpha_g(v) & 0} \matrix{\eins & 0 \\ 0 & w_g^*} = \matrix{0 & 0 \\ w_g\alpha_g(v) & 0}
\]
for all $g\in G$. In particular, we get $w_g = v\alpha_g(v^*)$ for all $g\in G$, which shows that $\set{w_g}_{g\in G}$ is a coboundary. This proves the claim
\end{proof}

The proof of Lemma \ref{O2 coboundary} was inspired in large part by the proof of \cite[Proposition 2.5]{Izumi04}.
In fact, the full power of that (potentially very useful) result can be recovered in the saturated setting for amenable group actions, so let us record it.
The proof is completely analogous as in \cite{Izumi04} by making use of Lemma \ref{fixed point}.
However, since we will have no use of this general result in this paper, we omit the proof.

\begin{theorem}
Let D be a unital, simple, purely infinite \cstar-algebra, $G$ a countable, discrete, amenable group and let $\alpha: G\curvearrowright D$ be a saturated and pointwise outer action.
Then there exists a one-to-one correspondence between $H^1(G,\CU(D))$ and the set
\[
\ker(K_0(\iota)) + [\eins] = \set{ x \in K_0(D^\alpha) ~|~ K_0(\iota)(x) = [\eins] },
\]
where $\iota$ is the inclusion map from $D^\alpha$ into $D$.
\end{theorem}

The following definition will serve a similar purpose as \cite[Definition 6.3]{GoldsteinIzumi11}.\footnote{We note that it is not a direct generalization, as we do not a priori restrict ourselves necessarily to unitary representations that induce embeddings of the group \cstar-algebra. Instead this assumption will appear in several intermediate proofs.}

\begin{defi}
Let $B$ be a unital \cstar-algebra and $G$ a countable, discrete group.
We say that a unitary representation $u: G\to\CU(B)$ is $K$-trivial, if the $KL$-class of the induced unital $*$-homomorphism $u: \cstar(G)\to B$ is equal to the $KL$-class of the canonical character $\chi: \cstar(G)\to\IC\subseteq B$.

If moreover these maps have the same $KK$-classes, we will say that $u$ is $KK$-trivial.
\end{defi}

\begin{rem}
Note that group \cstar-algebras of amenable groups satisfy the UCT due to Tu \cite{Tu99}, and thus Dadarlat-Loring's UMCT \cite{DadarlatLoring96} yields a natural isomorphism between $KL(\cstar(G), B)$ and $\operatorname{Hom}_\Lambda\big( \underline{K}(\cstar(G)), \underline{K}(B) \big)$.
In particular, a unitary representation $u: G\to\CU(B)$ is $K$-trivial if and only if the associated $*$-homomorphism $u: \cstar(G)\to B$ induces the same map on total $K$-theory as the character $\chi: \cstar(G)\to B$.
In terms of ordinary $K$-theory, this is equivalent to $K_*(u)=K_*(\chi)$ and $K_*(u\otimes\id_{\CO_{n+1}})=K_*(\chi\otimes\id_{\CO_{n+1}})$ for all $n\geq 2$.
\end{rem}

\begin{lemma} \label{coboundary}
Let $G$ be a discrete, countable and amenable group.
Let $D$ be a unital, simple and purely infinite \cstar-algebra and let $\alpha: G\curvearrowright D$ be a saturated and pointwise outer action.
Let $B$ be a unital \cstar-algebra and $u: G\to\CU(B)$ a $K$-trivial unitary representation that induces an embedding from $\cstar(G)$ to $B$. 
Let $\nu: B\to D^\alpha$ be a unital embedding. Then the collection $\set{\nu(u_g)}_{g\in G}$, viewed as an $\alpha$-cocycle in $D$, is a coboundary. 
\end{lemma}
\begin{proof}
Choose a unital embedding $\mu: \cstar(G)\to\CO_2$, which exists as $G$ is amenable.
By Lemma \ref{fixed point}, the fixed point algebra $D^\alpha$ is simple and purely infinite, and therefore we find a non-zero $*$-homom\-orphism $\iota': \CO_2\to D^\alpha$.
Then we obtain a unital $*$-homomorphism $\iota: \CO_2\oplus\IC\to D^\alpha$ via $\iota(x\oplus\lambda)=\iota'(x)+\lambda\cdot(\eins-\iota'(\eins))$.
By construction, the unital $*$-homomorphisms
\[
\nu\circ u : \cstar(G) \to D^\alpha
\]
and
\[
\iota\circ(\mu\oplus\chi): \cstar(G) \to D^\alpha
\]
are both faithful.
Since $\CO_2$ is $KK$-trivial and $u$ was assumed to be a $K$-trivial unitary represention, it follows that their $KL$-classes are the same.
By classification theory (see \cite[Theorem 3.14]{Lin05}), they are thus approximately unitarily equivalent.
As $D^\alpha$ is a saturated \cstar-algebra by Proposition \ref{semi-sat-permanence}\ref{semi-sat-permanence:2}, they are in fact unitarily equivalent.

Thus
\[
\nu\circ u = \ad(U)\circ\iota\circ(\mu\oplus\chi)
\]
for some unitary $U\in \CU(D^\alpha)$.
This allows us to reduce the problem to a special case via replacing $B$ by $\CO_2\oplus\IC$, $u$ by $(\mu\oplus\chi)$, and $\nu$ by $\ad(U)\circ\iota$.

Now denote $p=\nu(\eins_{\CO_2}\oplus 0_\IC)$ and $w: G\to\CU(\CO_2)$ the representation that arises from $u$ after projecting onto the first summand $\CO_2$.
The corner system $(pDp, \alpha|_{pDp})$ is pointwise outer and saturated by Proposition \ref{semi-sat-permanence}\ref{semi-sat-permanence:1} and Corollary \ref{spi-sat}.
As the underlying \cstar-algebra is simple and purely infinite, we are in the situation of Lemma \ref{O2 coboundary}.
It follows that the collection $\set{\nu(w_g)}_{g\in G}$, viewed as an $\alpha$-cocycle in $pDp$, is a coboundary.
So there exists a unitary $v_0\in pDp$ with $\nu(w_g))=v_0\alpha_g(v_0^*)$ for all $g\in G$.
Then $v=v_0+\eins-p$ is a unitary in $D$ with $\nu(u_g) = v\alpha_g(v^*)$ for all $g\in G$.
\end{proof}

\begin{lemma} \label{O embeddings}
Let $G$ be a discrete, countable and amenable group. 
Let $D$ be a unital, simple and purely infinite \cstar-algebra and let $\alpha: G\curvearrowright D$ be a saturated and pointwise outer action.
Let $B$ be a unital \cstar-algebra that embeds unitally into $D^\alpha$.
Then for any $K$-trivial unitary representation $u: G\to\CU(B)$ inducing an embedding from $\cstar(G)$ to $B$, there exists a unital and equivariant embedding from $(B, \ad(u))$ to $(D,\alpha)$.
\end{lemma}
\begin{proof}
Let $\mu: B\to D^\alpha$ be a unital embedding.
By Lemma \ref{coboundary}, we can find a unitary $v\in D$ with $\mu(u_g) = v\alpha_g(v^*)$ for all $g\in G$. Therefore, the map $\mu_0=\ad(v^*)\circ\mu$ yields a unital embedding from $B$ into $D$ satisfying
\[
(\alpha_g\circ\mu_0)(x) = \alpha_g(v^*) \mu(x) \alpha_g(v) = v^*\mu(u_g)\mu(x)\mu(u_g^*)v = \big( \mu_0\circ\ad(u_g) \big)(x)
\]
for all $g\in G$ and $x\in B$. This shows our claim.
\end{proof}

\begin{cor} \label{model-sat-embedding}
Let $G$ be a discrete, countable and amenable group.
Let $D$ be a unital, simple and purely infinite \cstar-algebra and let $\alpha: G\curvearrowright D$ be a saturated and pointwise outer action.
Let $u: G\to\CU(\CO_\infty)$ be a $K$-trivial unitary representation.
Then there exists an equivariant and unital $*$-homomorphism from $(\CO_\infty,\ad(u))$ to $(D,\alpha)$.
\end{cor}
\begin{proof}
As one sees further below in Example \ref{the-model-Oinf}, there certainly exist $K$-trivial unitary representations $u': G\to\CU(\CO_\infty)$ that induce embeddings from $\cstar(G)$ to $\CO_\infty$.
By replacing $u$ with $u\otimes u': G\to\CU(\CO_\infty\otimes\CO_\infty)$ and picking some isomorphism $\CO_\infty\otimes\CO_\infty\cong\CO_\infty$, we see that we may assume $u$ to induce an embedding of $\cstar(G)$ to $\CO_\infty$.

Then the claim follows from Lemmas \ref{fixed point} and \ref{O embeddings}, and the fact that $\CO_\infty$ embeds unitally into every unital, simple and purely infinite \cstar-algebra \cite{Cuntz77}.
\end{proof}


\section{Absorption of the model action on $\CO_\infty$}

In this section, we will prove our first main result, which asserts that pointwise outer amenable group actions on Kirchberg algebras absorb certain model actions on $\CO_\infty$ tensorially.

\begin{notae}
Using Definition \ref{Kirchberg nota}, we abbreviate $F_\omega(A)=F(A,A_\omega)$. In view of Remark \ref{product map}, every automorphism $\phi\in\Aut(A)$ gives rise to an automorphism $\tilde{\phi}_\omega$ on $F_\omega(A)$ in a natural way. 
\end{notae}

The following observation is a generalization of an observation \cite[Lemma 2]{Nakamura99} by Nakamura and has key importance:

\begin{theorem} \label{central seq outer}
Let $A$ be a separable and simple \cstar-algebra. Let $\phi\in\Aut(A)$ be an outer automorphism. For any free filter $\omega$ on $\IN$, the induced automorphism $\tilde{\phi}_\omega$ on $F_\omega(A)$ is outer.
\end{theorem}
\begin{proof}
By a result of Kishimoto \cite[Lemma 2.1]{Kishimoto81}, there exists an irreducible representation $\pi$ of $A$ such that $\pi$ and $\pi\circ\phi$ are disjoint. Let $\rho: A\to\CL(H)$ denote the universal representation of $A$.
The central cover $c(\pi)\in Z(\rho(A)'')$ of $\pi$ then satisfies $c(\pi)\phi(c(\pi))=0$. Choose some $\xi\in H$ of norm one with $\| \big( c(\pi)-\phi(c(\pi) \big)\xi\|=1$.
Applying a result of Akemann-Pedersen \cite[Lemma 1.1]{AkemannPedersen79}, we can find a bounded net $x_\lambda\in A$ with
\[
\|[x_\lambda,a]\| = \|[x_\lambda-c(\pi),a]\| \stackrel{\lambda\to\infty}{\longrightarrow} 0
\]
for every $a\in A$, and such that $x_\lambda$ converges to $c(\pi)$ in the strong operator topology. 
Since $\rho$ is non-degenerate, we can find a positive contraction $b\in A$ with $\| \rho(b)\xi-\xi\|< 1/4$. By triangle inequality, we therefore get that $\|\big( c(\pi)-\phi(c(\pi)) \big) \rho(b) \xi\|> \frac{1}{2}$.
 But then there exists $\lambda_0$ such that for every $\lambda\geq\lambda_0$ we have $\|\rho((x_\lambda-\phi(x_\lambda))b)\xi\|\geq 1/2$, which implies $\|(x_\lambda-\phi(x_\lambda))b\|\geq 1/2$.
 Using that $A$ is separable, we can construct a bounded sequence $x_n\in A$ such that
\[
\|[x_n,a]\|\stackrel{n\to\infty}{\longrightarrow} 0\quad\text{and}\quad\liminf_{n\to\infty} \|(x_n-\phi(x_n))b\|\geq 1/2.
\]
The class of the sequence $(x_n)_n$ then yields an element $x\in F_\omega(A)$ with $\|x-\tilde{\phi}_\omega(x)\|\geq \|(x-\tilde{\phi}_\omega(x))b\|\geq 1/2$.
By virtue of a standard reindexation argument, one can see that $\tilde{\phi}_\omega\in\Aut(F_\omega(A))$ must automatically be outer, if it is non-trivial. This shows our claim.
\end{proof}

\begin{prop} \label{F(A)-sat}
Let $G$ be a countable and discrete group. Let $A$ be a separable \cstar-algebra and $(\alpha,w): G\curvearrowright A$ a cocycle action.
Let $\omega\in\beta\IN\setminus\IN$ be a free ultrafilter.
Then the induced action $\tilde{\alpha}_\omega: G\curvearrowright F_\omega(A)$ is semi-saturated.
If $A$ is moreover a Kirchberg algebra, then $\tilde{\alpha}_\omega$ is saturated and $F_\omega(A)$ is simple and purely infinite. 
\end{prop}
\begin{proof}
By the stability of the central sequence algebra\footnote{The way we apply it here is the combination of \cite[Remark 1.8]{Szabo16ssa} and \cite[Proposition 1.9]{BarlakSzabo15}. The general fact that the central sequence algebra is stable is due to Kirchberg \cite[Corollary 1.10]{Kirchberg04}, but the version stated in that paper does not apply directly in this equivariant context.} and the Packer-Raeburn stabilization trick \cite[Theorem 3.4]{PackerRaeburn89}, we have
\[
\big( F_\omega(A), \tilde{\alpha}_\omega \big) \cong \big( F_\omega(A\otimes\CK), \tilde{\delta}_\omega \big)
\]
for some genuine action $\delta: G\curvearrowright A\otimes\CK$. In particular, we may assume without loss of generality that the cocycle $w=\eins$ is trivial.
In this case, the claim follows from Example \ref{ex:ultra} and Proposition \ref{F(A)-saturated}.
If $A$ is moreover a Kirchberg algebra, then the rest of the claim follows from Theorem \ref{F() pis} and Corollary \ref{spi-sat}.\footnote{Originally, the proof of the latter claim is due to Kirchberg-Phillips \cite{KirchbergPhillips00}. See also \cite[Theorem 2.12]{Kirchberg04}.}
\end{proof}

We will now use the results from the second section and start obtaining our first main results.

\begin{prop} \label{fixed point csa}
Let $A$ be a Kirchberg algebra and $G$ a countable, discrete and amenable group. Let $(\alpha,w): G\curvearrowright A$ be a cocycle action.
Let $\omega$ be a free ultrafilter on $\IN$.
Then the fixed point algebra $F_\omega(A)^{\tilde{\alpha}_\omega}$ of the  induced action $\tilde{\alpha}_\omega: G\curvearrowright F_\omega(A)$ is simple and purely infinite.
\end{prop}
\begin{proof}
Using a known trick from \cite[Proof of Theorem 3.1]{Matsumoto88}, we will reduce the general case to the pointwise outer case.
Consider the normal subgroup $H=\set{ h\in G ~|~ \alpha_h~\text{is inner}}\subset G$.
It follows from Theorem \ref{central seq outer} that $H$ coincides with the kernel of $\tilde{\alpha}_\omega$.
The induced action $\beta: G/H\curvearrowright F_\omega(A),~ \beta_{gH}=\tilde{\alpha}_{\omega,g}$ is then pointwise outer, and moreover semi-saturated by Proposition \ref{semi-sat-permanence}\ref{semi-sat-permanence:4}. 
Since $F_\omega(A)$ is simple and purely infinite by Theorem \ref{F() pis}, it follows from Corollary \ref{spi-sat} that $\beta$ is saturated, and by Lemma \ref{fixed point} that $F_\omega(A)^{\beta}$ is simple and purely infinite.
Observing that the fixed point algebras of $\tilde{\alpha}_\omega$ and of $\beta$ coincide by definition, this shows our claim.
\end{proof}

\begin{theorem} \label{equ-Oinf-absorption}
Let $A$ be a Kirchberg algebra and $G$ a countable, discrete and amenable group. Let $(\alpha,w): G\curvearrowright A$ be a cocycle action. Then $(\alpha,w)$ is strongly cocycle conjugate to $(\alpha\otimes\id_{\CO_\infty},w\otimes\eins)$.
\end{theorem}
\begin{proof}
By Proposition \ref{fixed point csa}, the fixed point algebra of $\tilde{\alpha}_\omega: G\curvearrowright F_\omega(A)$ is simple and purely infinite. In particular, we can find some unital embedding of $\CO_\infty$ into $F_\omega(A)^{\tilde{\alpha}_\omega}$. The claim follows from \cite[Corollary 3.7]{Szabo16ssa}.
\end{proof}

\begin{theorem} \label{model-absorption}
Let $A$ be a Kirchberg algebra and $G$ a countable, discrete and amenable group. Let $(\alpha,w): G\curvearrowright A$ be a pointwise outer cocycle action.
For every $n\in\IN$, let $u_n: G\to\CU(\CO_\infty)$ be a $K$-trivial unitary representation.
Then $(\alpha,w)$ is strongly cocycle conjugate to 
\[
\Big( \alpha\otimes\bigotimes_{n\in\IN} \ad(u_n), w\otimes\eins^{\otimes\infty} \Big): G\curvearrowright A\otimes\bigotimes_\IN\CO_\infty.
\]
\end{theorem}
\begin{proof}
We may assume without loss of generality that any unitary representation occuring in the sequence $u_n$ is occuring infinitely many times.
Then it follows from \cite[Proposition 5.3]{Szabo16ssa2} that
\[
\gamma=\bigotimes_{n\in\IN} \ad(u_n): G\curvearrowright \bigotimes_\IN\CO_\infty
\]
is a strongly self-absorbing action.
By the equivariant McDuff-type theorem \cite[Theorem 3.6]{Szabo16ssa} for strongly self-absorbing actions, it suffices to show that there exists an equivariant and unital $*$-homomorphism from $\Big( \bigotimes_\IN \CO_\infty,\bigotimes_{n\in\IN} \ad(u_n)\Big)$ into $\big( F_\omega(A), \tilde{\alpha}_\omega \big)$.
By a standard reindexation argument, it suffices to show that for every $n\in\IN$, there exists an equivariant and unital $*$-homomorphism from $\big( \CO_\infty, \ad(u_n)\big)$ into $\big( F_\omega(A), \tilde{\alpha}_\omega \big)$.
The action $\tilde{\alpha}_\omega$ is saturated by Proposition \ref{F(A)-sat}, and it follows from Theorem \ref{central seq outer} that it is pointwise outer.
As $F_\omega(A)$ is simple and purely infinite, the claim follows follows from Corollary \ref{model-sat-embedding}.
\end{proof}

\begin{example} \label{the-model-Oinf}
Let $G$ be a countable, discrete and amenable group.
Since $\cstar(G)$ unitally embeds into $\CO_2$, there exists a unitary representation $v: G\to\CU(\CO_2)$ that induces a $*$-monomorphism on $\cstar(G)$.
Fix some non-zero $*$-homomorphism $\iota:\CO_2\to\CO_\infty$.
Then $u_g=\iota(v_g)+\eins-\iota(\eins)$ yields a unitary representation $u: G\to\CU(\CO_\infty)$, which by construction is $K$-trivial (in fact $KK$-trivial) and still induces a $*$-monomorphism on $\cstar(G)$.
Then we may consider the action
\[
\gamma=\bigotimes_\IN\ad(u): G\curvearrowright\bigotimes_\IN\CO_\infty\cong\CO_\infty,
\]
which for our purposes can be regarded as {\it the} standard model action on $\CO_\infty$. Such an action is strongly self-absorbing by \cite[Proposition 5.3]{Szabo16ssa2}.
\end{example}

The following two statement are then direct consequences of Theorem \ref{model-absorption}:

\begin{cor} \label{the-model-absorption}
Let $A$ be a Kirchberg algebra and $G$ a countable, discrete and amenable group.
Let $(\alpha,w): G\curvearrowright A$ be a pointwise outer cocycle action.
Let $\gamma$ be the action constructed in Example \ref{the-model-Oinf}.
Then $(\alpha,w)$ is strongly cocycle conjugate to $(\alpha\otimes\gamma,w\otimes\eins)$.
\end{cor}

\begin{cor} \label{model-uniqueness}
Let $G$ be a countable, discrete and amenable group.
For every $n\in\IN$ and $i=0,1$, let $u_n^{(i)}: G\to\CU(\CO_\infty)$ be a $K$-trivial unitary representation.
If the two actions
\[
\gamma^{(i)}=\bigotimes_{n\in\IN}\ad(u_n^{(i)}): G\curvearrowright\bigotimes_\IN\CO_\infty, \quad i=1,2
\]
are both pointwise outer, then they are strongly cocycle conjugate.
In particular, the strong cocycle conjugacy class of the action in Example \ref{the-model-Oinf} does not depend on the choices to construct it.
\end{cor}


\section{More about the model action on $\CO_\infty$}

Consider the model $G$-action $\gamma$ on $\CO_\infty$ considered in Example \ref{the-model-Oinf}.
In this section, we investigate under what abstract conditions a pointwise outer $G$-action $\alpha$ on a strongly self-absorbing Kirchberg algebra $\CD\cong\CO_\infty\otimes\CD$ is strongly cocycle conjugate to $\gamma\otimes\id_\CD$.
We note that some key facts observed in this section, which our proof relies on, were observed prior by Izumi-Matui in \cite[Section 6]{IzumiMatui10}. 

In order to understand what is necessary for such an action $\alpha$ to be strongly cocycle conjugate to one of the model actions, we first need some preparation:

\begin{prop} \label{model-KK-1}
Let $u: G\to\CU(\CO_\infty)$ be a $KK$-trivial unitary representation. 
Then the canonical inclusion $\cstar(G)\subset \CO_\infty\rtimes_{\ad(u)} G$ induces a $KK$-equivalence.
\end{prop}
\begin{proof}
We have a natural isomorphism
\[
\psi: \CO_\infty\rtimes_{\ad(u)} G\to \CO_\infty\otimes \cstar(G)
\]
via
\[
\CO_\infty\ni a\mapsto a\otimes\eins ~\text{ and }~ \lambda^{\ad(u)}_g\mapsto u_g\otimes\lambda_g.
\]
As $u$ is $KK$-trivial, this implies
\[
KK(g\mapsto u_g\otimes\lambda_g) = KK(g\mapsto \eins_{\CO_\infty}\otimes\lambda_g)=KK(\eins_{\CO_\infty}\otimes\id_{\mathrm{C}^*(G)}).
\]
The third $KK$-element is represented by the second-factor embedding of $\cstar(G)$ into $\CO_\infty\otimes\cstar(G)$, and is therefore a $KK$-equivalence.
But this implies that the $*$-homomorphism $\mu: \cstar(G)\to\CO_\infty\otimes \cstar(G)$ induced by the unitary representation $g\mapsto u_g\otimes\lambda_g$ is also a $KK$-equivalence.
Note that $\psi^{-1}\circ\mu: \cstar(G)\to \CO_\infty\rtimes_{\ad(u)} G$ is nothing but the canonical inclusion, and therefore also induces a $KK$-equivalence.
\end{proof}

\begin{cor} \label{model-KK-2}
If $\gamma$ is any model action as in Example \ref{the-model-Oinf}, then the canonical inclusion $\cstar(G)\subset\CO_\infty\rtimes_\gamma G$ induces a $KK$-equivalence.
\end{cor}
\begin{proof}
By our assumptions on $\gamma$, there exists a $KK$-trivial unitary representation $u: G \to \CU(\CO_\infty)$ such that 
\[
(\CO_\infty,\gamma) \cong \lim_{\longrightarrow} \big( \CO_\infty^{\otimes n}, \ad(u^{\otimes n}) \big).
\]
We thus have some commutative diagram of the form
\[
\xymatrix@R+5mm@C+5mm{
\CO_\infty\rtimes_{\ad(u)} G \ar[r] & (\CO_\infty^{\otimes 2})\rtimes_{\ad(u^{\otimes 2})} G \ar[r] & \cdots\cdots \ar[r] & \CO_\infty\rtimes_\gamma G \\
\cstar(G) \ar[u] \ar[ur] \ar[urr] \ar[urrr] &&&
}
\]
where the maps from bottom to top are the canonical inclusions.
By Proposition \ref{model-KK-1}, every such map into one of the building blocks induces a $KK$-equivalence, as each representation $u^{\otimes n}$ is $KK$-trivial.
But then the inclusion $\cstar(G)\subset\CO_\infty\rtimes_\alpha G$ must also induce a $KK$-equivalence by virtue of this inductive limit structure and \cite[Proposition 2.4]{Dadarlat09}.
\end{proof}

\begin{defi}[see {\cite[Definition 1.4(iv)]{Szabo17ssa3}}]
Let $G$ be a countable, discrete group.
Let $(A,\alpha)$ and $(B,\beta)$ be two unital $G$-\cstar-dynamical systems.
We say that $\alpha$ and $\beta$ are very strongly cocycle conjugate, if there exists an isomorphism $\phi: A\to B$, a $\beta$-cocycle $\set{w_g}_{g\in G}\subset\CU(B)$ such that $\beta^w=\phi\circ\alpha\circ\phi^{-1}$, and such that moreover $w_g=\lim_{t\to\infty} v_t\alpha_g(v_t^*)$ for all $g\in G$, where $\set{v_t}_{t\geq 0}\subset\CU(B)$ is a continuous one-parameter family of unitaries with $v_0=\eins$.
We write $\alpha\vscc\beta$. A $\beta$-cocycle $\set{w_g}_{g\in G}$ with this property is called an asymptotic coboundary.
\end{defi}

\begin{prop} \label{model-KK-condition}
Let $\gamma$ be a model action as in Example \ref{the-model-Oinf}.
Let $\alpha: G\curvearrowright\CO_\infty$ be an action. 
\begin{enumerate}[label=\textup{(\arabic*)},leftmargin=*]
\item Suppose $\alpha\scc\gamma$. Then the canonical inclusion
\[
\mathrm{C}^*(G)\subset\CO_\infty\rtimes_\alpha G
\]
induces a $KL$-equivalence. \label{model-KK:1}
\item Suppose $\alpha\vscc\gamma$. Then the canonical inclusion
\[
\mathrm{C}^*(G)\subset\CO_\infty\rtimes_\alpha G
\]
induces a $KK$-equivalence. \label{model-KK:2}
\end{enumerate} 
\end{prop}
\begin{proof}
Clearly it suffices to consider the case where $\gamma$ is a cocycle perturbation of $\alpha$ by an approximate or asymptotic coboundary, respectively.
So let $\set{w_g}_{g\in G}$ be an $\alpha$-cocycle. The canonical isomorphism $\psi: \CO_\infty\rtimes_{\alpha^w} G\to \CO_\infty\rtimes_\alpha G$ is given by the formulas $\psi|_{\CO_\infty}=\id_{\CO_\infty}$ and
\[
\psi(\lambda_g^{\alpha^w})=w_g\lambda_g^\alpha
\]
for all $g\in G$. Let us consider the following diagram:
\begin{equation} \label{eq:cp-diagram}
\xymatrix{
& \mathrm{C}^*(G) \ar[ld] \ar[rd] & \\
\CO_\infty\rtimes_{\alpha^w} G \ar[rr]^\psi && \CO_\infty\rtimes_\alpha G
}
\end{equation}

\ref{model-KK:1}: Suppose that $\set{w_g}_{g\in G}$ is an approximate coboundary, i.e., there exists a sequence of unitaries $v_n\in\CO_\infty$ with
\[
w_g=\lim_{n\to\infty} v_n\alpha_g(v_n)^*\quad\text{for all}\ g\in G.
\]
Thus
\[
\psi(\lambda_g^{\alpha^w})=w_g\lambda_g^\alpha = \lim_{n\to\infty} v_n\alpha_g(v_n^*)\lambda_g^\alpha = \lim_{n\to\infty} v_n \lambda_g^\alpha v_n^*
\]
for all $g\in G$. In particular, the diagram \eqref{eq:cp-diagram}
commutes up to approximate unitary equivalence.
So this diagram commutes on the level of $KL$-theory and the left vertical arrow induces a $KL$-equivalence if and only if the right vertical arrow induces a $KL$-equivalence. The claim then follows from Corollary \ref{model-KK-2}.

\ref{model-KK:2}: Suppose that $\set{w_g}_{g\in G}$ is an asymptotic coboundary, i.e., there exists a continuous map $v: [0,\infty)\to\CU(\CO_\infty)$ with $v_0=\eins$ and
\[
w_g=\lim_{t\to\infty} v_t\alpha_g(v_t)^*\quad\text{for all}\ g\in G.
\]
Thus
\[
\psi(\lambda_g^{\alpha^w})=w_g\lambda_g^\alpha = \lim_{t\to\infty} v_t\alpha_g(v_t^*)\lambda_g^\alpha = \lim_{t\to\infty} v_t \lambda_g^\alpha v_t^*
\]
for all $g\in G$. In particular, the diagram \eqref{eq:cp-diagram}
commutes up to homotopy.
So this diagram commutes on the level of $KK$-theory and the left vertical arrow induces a $KK$-equivalence if and only if the right vertical arrow induces a $KK$-equivalence. The claim then follows from Corollary \ref{model-KK-2}.
\end{proof}

The following shows that the two cases in Proposition \ref{model-KK-condition} are in fact equivalent. We will use this in the proof of the main result (Theorem \ref{approx-rep-ssa-uniqueness}) of this section in order to obtain the improvement from $KL$-equivalence to $KK$-equivalence.

\begin{prop} \label{KK-or-KL}
Let $G$ be a countable, discrete, amenable group and $\CD$ a strongly self-absorbing Kirchberg algebra.
Let $\gamma: G\curvearrowright\CD$ be a semi-strongly self-absorbing action. 
\begin{enumerate}[label=\textup{(\arabic*)},leftmargin=*]
\item If $\beta: G\curvearrowright B$ is an action on a separable, unital \cstar-algebra, then $\beta\scc\beta\otimes\gamma$ if and only if $\beta\vscc\beta\otimes\gamma$. \label{KK-or-KL:1}
\item For every action $\alpha: G\curvearrowright\CD$, one has $\alpha\scc\gamma$ if and only if $\alpha\vscc\gamma$. \label{KK-or-KL:2}
\end{enumerate}
\end{prop}
\begin{proof}
\ref{KK-or-KL:1}: The action $\gamma$ is equivariantly $\CO_\infty$-absorbing by Theorem \ref{equ-Oinf-absorption}, and thus unitarily regular by \cite[Proposition 1.19]{Szabo16ssa2}.
Then the claim follows from \cite[Theorem 2.2]{Szabo17ssa3}.

\ref{KK-or-KL:2}: Suppose $\alpha\scc\gamma$.
Since $\gamma$ is semi-strongly self-absorbing, so is $\alpha$, and we therefore have $\alpha\scc\alpha\otimes\gamma\scc\gamma$.  Applying \ref{KK-or-KL:1} twice, we get $\alpha\vscc\alpha\otimes\gamma\vscc\gamma$.
\end{proof}

\begin{notae}
For the rest of this section, let us denote by $j_\alpha: \cstar(G)\to A\rtimes_\alpha G$ the canonical embedding for an action $\alpha: G\curvearrowright A$ of a discrete group on a unital \cstar-algebra.
\end{notae}

\begin{prop} \label{embedding-into-model}
Let $G$ be a countable, discrete, amenable group.
Let $\CD$ be a strongly self-absorbing Kirchberg algebra and let $\alpha: G\curvearrowright\CD$ be a pointwise outer action.
Suppose that the crossed product $\CD\rtimes_\alpha G$ is $\CD$-stable and that the canonical embedding $\id_\CD\otimes j_\alpha: \CD\otimes\cstar(G)\to \CD\otimes(\CD\rtimes_\alpha G)$ is a $KL$-equivalence.
Let $u: G\to\CU(\CD)$ be a $K$-trivial unitary representation yielding a $*$-monomorphism from $\mathrm{C}^*(G)$ to $\CD$.
Then there exist approximately equivariant and unital $*$-homom\-orphisms from $(\CD,\alpha)$ to $(\CD,\ad(u))$.
\end{prop}
\begin{proof}
We consider the commutative diagram
\[
\xymatrix@R+5mm{
\mathrm{C}^*(G) \ar[rrd]_{j_\alpha} \ar@/^1.75pc/[rrrr]^{\eins\otimes\id_{\cstar(G)}} && \CD \ar[rrd]^{\id_\CD\otimes\eins} \ar[rr]^{\id_\CD\otimes\eins} && \CD\otimes \mathrm{C}^*(G) \ar[d]^{\id_\CD\otimes j_\alpha} \\
&& \CD\rtimes_\alpha G \ar@{-->}[u]_\phi \ar[rr]^{\eins\otimes\id}  && \CD\otimes(\CD\rtimes_\alpha G)
}
\]
consisting of the obvious embeddings.
By what we know, the right vertical map and the bottom horizontal map each induce a $KL$-equivalence. Let $\chi: \mathrm{C}^*(G)\to\IC$ be the canonical character.
Then $\id_\CD\otimes\chi$ yields a left-inverse for $\id_\CD\otimes\eins_{\cstar(G)}$. In particular, we obtain a $KL$-element
\[
\kappa=KL(\eins\otimes\id_{\CD\rtimes_\alpha G})\otimes KL(\id_\CD\otimes j_\alpha)^{-1}\otimes KL(\id_\CD\otimes\chi) \in KL(\CD\rtimes_\alpha G,\CD).
\]
By Kirchberg-Phillips classification \cite{KirchbergC, Phillips00}, this $KL$-element is represented by a unital $*$-monomorphism $\phi: \CD\rtimes_\alpha G\to\CD$, which still makes the above diagram commutative on the level of $KL$-theory.
In particular, we have
\[
KL(\phi\circ j_\alpha) = KL(\eins_\CD\otimes\chi) \in KL(\mathrm{C}^*(G),\CD).
\]
Notice that the unital $*$-monomorphism $\mu: \mathrm{C}^*(G)\to\CD$ induced by $u: G\to\CD$ also satisfies
\[
KL(\mu) = KL(\eins_\CD\otimes\chi)
\]
because $u$ was assumed to be $K$-trivial.

From the uniqueness theorem with respect to $KL$-theory in Kirchberg-Phillips classification (see \cite[Theorem 3.14]{Lin05}), it follows that $\mu$ and $\phi\circ j_\alpha$ are approximately unitarily equivalent.
In particular, we find a sequence of unitaries $v_n\in\CU(\CD)$ such that the $*$-homomorphisms
\[
\psi_n=\ad(v_n)\circ\phi: \CD\rtimes_\alpha G\to\CD
\]
satisfy
\[
\psi_n(\lambda^\alpha_g) \stackrel{n\to\infty}{\longrightarrow} u_g\quad\text{for all}~g\in G.
\]
But this implies for all $a\in\CD\subset\CD\rtimes_\alpha G$ and $g\in G$ that
\[
\renewcommand\arraystretch{1.3}
\begin{array}{cl}
\multicolumn{2}{l}{ \| \big( \ad(u_g)\circ\psi_n \big)(a)- \big( \psi_n\circ\alpha_g \big)(a)\| } \\
=& \| \big( \ad(u_g)\circ\psi_n \big)(a)- \big( \ad(\psi_n(\lambda^\alpha_g))\circ\psi_n \big)(a)\| \\
\stackrel{n\to\infty}{\longrightarrow} & 0.
\end{array}
\]
In particular, if we denote by $\iota_\alpha: \CD\to \CD\rtimes_\alpha G$ the inclusion map, then the sequence of $*$-homomorphisms $\psi_n\circ \iota_\alpha$ yields an approximately equivariant and unital embedding from $(\CD,\alpha)$ into $(\CD,\ad(u))$.
\end{proof}

The following observation is a variant of \cite[Remark 4.9]{IzumiMatui10}.
We note that the unitality assumption for the involved \cstar-algebras is not necessary, and can be dropped with some minor additional effort in the proof.
We will only apply this in the unital case.

\begin{prop} \label{approx-rep-uniqueness}
Let $G$ be a countable, discrete group.
Let $A$ and $B$ be two separable, unital \cstar-algebras.
Let $\alpha: G\curvearrowright A$ and $\beta: G\curvearrowright B$ be two actions, and assume that $\beta$ is approximately representable.
Let $\phi_1, \phi_2: (A,\alpha)\to (B,\beta)$ be two equivariant and unital $*$-homomorphisms.
Then $\phi_1$ and $\phi_2$ are approximately $G$-unitarily equivalent if and only if the two unital $*$-homomorphisms
\[
\phi_i\rtimes G: A\rtimes_\alpha G \to B\rtimes_\beta G,\quad i=1,2
\]
are approximately unitarily equivalent.
\end{prop}
\begin{proof}
The 'only if' part is trivial, so let us show the 'if' part.

Let $v_n\in B\rtimes_\beta G$ be a sequence of unitaries such that
\[
\ad(v_n)\circ (\phi_1\rtimes G) \stackrel{n\to\infty}{\longrightarrow} \phi_2\rtimes G.
\]
Recall from \cite[Section 4]{BarlakSzabo15} that $\beta$ is approximately representable if and only if the $G$-equivariant inclusion $(B,\beta)\to (B\rtimes_\beta G, \ad(\lambda^{\beta}))$ is equivariantly sequentially split.
Let $\omega$ be a free ultrafilter on $\IN$. For $i=1,2$, we then have a commutative diagram of $G$-equivariant $*$-homomorphisms of the form
\[
\xymatrix@C-3mm{
(A,\alpha) \ar[rd] \ar[rr]^{\phi_i} && (B,\beta) \ar[rd] \ar[rr] && (B_\omega,\beta_\omega) \\
& \big( A\rtimes_\alpha G, \ad(\lambda^\alpha) \big) \ar[rr]^{\phi_i\rtimes G} && \big( B\rtimes_\beta G, \ad(\lambda^\beta) \big) \ar[ru]_\psi &
}
\]
The maps without names are the canonical ones. Since the maps $\phi_i\rtimes G$ do the same on the canonical copy of $\mathrm{C}^*(G)\subset A\rtimes_\alpha G$, it follows that $\|[v_n,\lambda^\beta_g]\|\to 0$ for all $g\in G$.
By the commutativity of the above diagram, it follows for all $a\in A$ that
\[
\renewcommand\arraystretch{1.5}
\begin{array}{ccl}
\phi_2(a) &=& \big( \psi\circ(\phi_2\rtimes G) \big)(a) \\
&=& \dst\lim_{n\to\infty} \big( \ad(\psi(v_n))\circ\psi\circ(\phi_1\rtimes G) \big)(a) \\
&=& \dst\lim_{n\to\infty} \big(\ad(\psi(v_n))\circ\phi_1\big)(a).
\end{array}
\]
Moreover also
\[
\|\beta_{\omega,g}(\psi(v_n))-\psi(v_n)\| \leq \| \ad(\lambda^\beta_g)(v_n)-v_n\| \stackrel{n\to\infty}{\longrightarrow} 0
\]
for all $g\in G$. By representing each $\psi(v_n)\in \CU(B_\omega)$ via a sequence of unitaries in $B$, we can apply a standard reindexation argument and obtain a sequence of unitaries $u_n\in\CU(B)$ satisfying
\[
\|u_n-\beta_g(u_n)\|\stackrel{n\to\infty}{\longrightarrow} 0
\]
and
\[
\phi_2(a) = \lim_{n\to\infty} (\ad(u_n)\circ\phi_1)(a)
\]
for all $a\in A$ and $g\in G$. This finishes the proof.
\end{proof}

\begin{prop} \label{approx-rep-D-stable}
Let $G$ be a countable, discrete group. Let $A$ be a separable \cstar-algebra and $\alpha: G\curvearrowright A$ an action.
Let $\CD$ be a strongly self-absorbing \cstar-algebra. Assume that $\alpha$ is approximately representable. Then $A\rtimes_\alpha G$ is $\CD$-stable if and only if $\alpha\cc\alpha\otimes\id_\CD$.
\end{prop}
\begin{proof}
Since the ``if'' part is trivial, let us show the ``only if'' part. Assume that $A\rtimes_\alpha G$ is $\CD$-stable.
By \cite[Theorem 2.6]{BarlakSzabo15}\footnote{This is nothing but a reformulation of Toms-Winter's criterion \cite[Theorem 2.3]{TomsWinter07}.}, this means that
\[
\id\otimes\eins: A\rtimes_\alpha G \to (A\rtimes_\alpha G)\otimes\CD
\]
is sequentially split. But then we see that
\[
\id\otimes\eins: \big( A\rtimes_\alpha G, \ad(\lambda^\alpha)\big) \to \big( ( A\rtimes_\alpha G)\otimes\CD, \ad(\lambda^\alpha)\otimes\id_\CD \big)
\]
is $G$-equivariantly sequentially split. Since $\alpha$ is approximately representable, it follows (cf.\ \cite[Subsection 4.4]{BarlakSzabo15}) that the canonical inclusion
\[
\iota_\alpha : (A,\alpha)\to (A\rtimes_\alpha G, \ad(\lambda^\alpha))
\]
is sequentially split. Hence the composition
\[
(\id\otimes\eins)\circ\iota_\alpha = (\iota_\alpha\otimes\id_\CD)\circ (\id_A\otimes\eins)
\]
is also $G$-equivariantly sequentially split by \cite[Proposition 3.7]{BarlakSzabo15}. This in particular implies that
\[
\id_A\otimes\eins: (A,\alpha) \to (A\otimes\CD,\alpha\otimes\id_\CD)
\]
is sequentially split, and so by \cite[Theorem 4.29]{BarlakSzabo15}, we get $\alpha\cc\alpha\otimes\id_\CD$.
\end{proof}

The following is the main result of this section. 
Note that the $KL$-assumption below is necessary by Proposition \ref{model-KK-condition}.

\begin{theorem} \label{approx-rep-ssa-uniqueness}
Let $G$ be a countable, discrete, amenable group. 
Let $\CD$ be a strongly self-absorbing Kirchberg algebra and $\alpha: G\curvearrowright\CD$ a pointwise outer action.
Suppose that the crossed product $\CD\rtimes_\alpha G$ is $\CD$-stable, and that the canonical embedding $\CD\otimes\cstar(G)\to \CD\otimes(\CD\rtimes_\alpha G)$ is a $KL$-equivalence.
Let $\gamma: G\curvearrowright\CO_\infty$ be an action as in Example \ref{the-model-Oinf}. Then the following are equivalent:
\begin{enumerate}[label=\textup{(\roman*)}, leftmargin=*]
\item $\alpha$ is strongly cocycle conjugate to $\id_\CD\otimes\gamma$. \label{ar-ssa-u:1}
\item $\alpha$ is cocycle conjugate to $\id_\CD\otimes\gamma$; \label{ar-ssa-u:2}
\item $\alpha$ is approximately representable; \label{ar-ssa-u:3}
\item $\alpha$ has approximately $G$-inner half-flip and $\alpha\cc\alpha\otimes\id_\CD$; \label{ar-ssa-u:4}
\item $\alpha$ is semi-strongly self-absorbing and $\alpha\cc\alpha\otimes\id_\CD$. \label{ar-ssa-u:5}
\end{enumerate}
Moreover, in this case the inclusion $\CD\otimes\cstar(G)\subset \CD\otimes(\CD\rtimes_\alpha G)$ induces a $KK$-equivalence and $\alpha$ is very strongly cocycle conjugate to $\id_\CD\otimes\gamma$.
\end{theorem}
\begin{proof}
The implication \ref{ar-ssa-u:1}$\implies$\ref{ar-ssa-u:2} is trivial, and \ref{ar-ssa-u:2}$\implies$\ref{ar-ssa-u:3} follows from the approximate representability of $\gamma$ and the well-known fact that approximate representability is closed under cocycle conjugacy.
The implications \ref{ar-ssa-u:1}$\implies$\ref{ar-ssa-u:5}$\implies$\ref{ar-ssa-u:4} are also trivial; see \cite[Theorem 4.6]{Szabo16ssa}. Let us now show the other implications.

\ref{ar-ssa-u:3}$\implies$\ref{ar-ssa-u:4}: It follows from our assumptions and Proposition \ref{approx-rep-D-stable} that $\alpha\cc\alpha\otimes\id_\CD$.
Denote
\[
\rho_l= (\id_\CD\otimes\eins)\rtimes G: \CD\rtimes_\alpha G \to (\CD\otimes\CD)\rtimes_{\alpha\otimes\alpha} G
\]
and
\[
\rho_r= (\eins\otimes\id_\CD)\rtimes G: \CD\rtimes_\alpha G \to (\CD\otimes\CD)\rtimes_{\alpha\otimes\alpha} G.
\]
Consider the following commutative diagram:
\[
\xymatrix@R+3mm@C+8mm{
& (\CD\otimes\CD)\rtimes_{\alpha\otimes\alpha} G \ar[r]^{\eins_\CD\otimes\id} & \CD\otimes\Big((\CD\otimes\CD)\rtimes_{\alpha\otimes\alpha} G \Big) \\
\CD\rtimes_\alpha G \ar[r]^{\eins_\CD\otimes\id} \ar[ru]_{\rho_l} \ar[rd]^{\rho_r} & \CD\otimes (\CD\rtimes_\alpha G) \ar[ru]_{\id_\CD\otimes\rho_l} 
\ar[rd]^{\id_\CD\otimes\rho_r} & \CD\otimes\mathrm{C}^*(G) \ar[l]_{\id_\CD\otimes j_\alpha} \\
& (\CD\otimes\CD)\rtimes_{\alpha\otimes\alpha} G \ar[r]^{\eins_\CD\otimes\id} & \CD\otimes\Big((\CD\otimes\CD)\rtimes_{\alpha\otimes\alpha} G \Big)
}
\]
We know that all the horizontal maps in this diagram induce $KL$-equivalences. It is moreover obvious that
\[
(\id_\CD\otimes\rho_l)\circ(\id_\CD\otimes j_\alpha)=(\id_\CD\otimes\rho_r)\circ(\id_\CD\otimes j_\alpha).
\]
This implies $KL(\rho_l)=KL(\rho_r)$ due to the above diagram.
By the $KL$-uniqueness in Kirchberg-Phillips classification (see \cite[Theorem 3.14]{Lin05}), it follows that $\rho_l$ and $\rho_r$ are approximately unitarily equivalent.
Since $\alpha$ is approximately representable, Proposition \ref{approx-rep-uniqueness} implies that $\id_\CD\otimes\eins_\CD$ and $\eins_\CD\otimes\id_\CD$ are approximately $G$-unitarily equivalent.
In other words, $\alpha$ has approximately $G$-inner half-flip.

\ref{ar-ssa-u:4}$\implies$\ref{ar-ssa-u:1}: By our assumption and Theorem \ref{model-absorption}, it follows that $\alpha\cc\alpha\otimes\id_\CD\otimes\gamma$.
Let $u: G\to\CU(\CO_\infty)$ be a $KK$-trivial unitary representation that induces a $*$-monomorphism from $\mathrm{C}^*(G)$ to $\CO_\infty$.
Without loss of generality, we may assume that $u$ is the unitary representation used to construct $\gamma$ in Example \ref{the-model-Oinf}. Then
\[
\id_\CD\otimes\gamma ~\cong~ \bigotimes_\IN \big( \id_\CD\otimes\ad(u) \big).
\]
By Proposition \ref{embedding-into-model}, $(\CD,\alpha)$ admits unital and approximately equivariant embeddings into $(\CD\otimes\CO_\infty,\id_\CD\otimes\ad(u))$.
In particular, $(\CD,\alpha)$ admits unital and approximately central and equivariant embeddings into $(\CD\otimes\CO_\infty,\id_\CD\otimes\gamma)$.
It then follows from \cite[Theorem 2.7]{Szabo16ssa} that $\id_\CD\otimes\gamma\scc\alpha\otimes\id_\CD\otimes\gamma$. 

The last part of the claim follows from Proposition \ref{KK-or-KL}\ref{KK-or-KL:2} and the argument analogous to Proposition \ref{model-KK-condition}\ref{model-KK:2}.
\end{proof}

\begin{cor} \label{abstract-models-Oinf}
Let $G$ be a countable, discrete and amenable group.
Let $\alpha: G\curvearrowright\CO_\infty$ be a pointwise outer action. 
Then $\alpha$ is strongly cocycle conjugate to one of the model actions in Example \ref{the-model-Oinf} if and only if the canonical inclusion $\mathrm{C}^*(G)\subset \CO_\infty\rtimes_\alpha G$ induces a $KL$-equivalence, and $\alpha$ is either approximately representable or semi-strongly self-absorbing.
If this is the case, then the inclusion in fact induces a $KK$-equivalence.
\end{cor}
\begin{proof}
The ``only if'' part is clear from Proposition \ref{model-KK-condition}\ref{model-KK:1}, so what is left to show is the ``if'' part.
Since the inclusion $\IC\subset\CO_\infty$ is a $KK$-equivalence, it is clear that the assumption implies that the inclusion $\CO_\infty\otimes\cstar(G)\subset\CO_\infty\otimes(\CO_\infty\rtimes_\alpha G)$ is a $KL$-equivalence.
Note that we have $\alpha\cc\alpha\otimes\id_{\CO_\infty}$ automatically by Theorem \ref{equ-Oinf-absorption}.
If $\alpha$ is now either approximately representable or semi-strongly self-absorbing, then the claim follows from Theorem \ref{approx-rep-ssa-uniqueness}.
\end{proof}

\begin{rem} \label{Z2-example}
The assumption on an action $\alpha$ that $\mathrm{C}^*(G)\subset\CO_\infty\rtimes_\alpha G$ induces a $KL$-equivalence has a certain dichotomy to it.
This assumption is necessary if $G$ has torsion.
For instance, if $p\in\CO_\infty$ is a projection with $[p]_0=2$ in the $K_0$-group, then the action
\[
\bigotimes_\IN \ad(2p-\eins): \IZ_2\curvearrowright\bigotimes_\IN\CO_\infty
\]
has a crossed product with $K$-theory different from $\mathrm{C}^*(\IZ_2)=\IC^2$. This is an observation of Izumi and follows from a direct computation.
On the other hand, for actions of groups such as $\IZ^N$, the redundance of this assumption has been a key step in works such as \cite[Lemma 5.3]{IzumiMatui10}.
A related observation plays an important role for strongly outer $\IZ^N$-actions on UHF algebras of infinite type, see \cite[Lemma 3.1]{Matui10}. 
\end{rem}

\begin{rem} \label{conj:ar-redundant}
It remains open whether the conditions listed in Theorem \ref{approx-rep-ssa-uniqueness}, in particular condition \ref{ar-ssa-u:3}, are redundant. 
It is even unclear whether it is redundant if one assumes that $\alpha$ is $KK^G$-equivalent to the trivial $G$-action on $\CD$. 
This latter condition is known to hold if $\alpha$ is a quasi-free action on $\CO_\infty$; see \cite[Section 4]{Pimsner97}. 
Indeed, the main result of \cite{GoldsteinIzumi11} shows that for finite groups, any faithful and quasi-free action on $\CO_\infty$ is conjugate to the model action in Example \ref{the-model-Oinf}. 
However, the proof uses certain techniquess that are not readily available for infinite groups. 
It would thus be natural to check whether any faithful and quasi-free action on $\CO_\infty$ is approximately representable. 
It should also be noted that there is no example of any group action of any exact group on $\CO_\infty$ or $\CO_2$ that is known not to be approximately representable. 
\end{rem}

In what comes next, we shall give a general argument, using Baum-Connes for amenable groups, that for groups without torsion, the initial assumptions in Theorem \ref{approx-rep-ssa-uniqueness} about crossed products are redundant.
The key fact that we get from Baum-Connes is the following:

\begin{theorem}[see {\cite[Theorem 8.5]{MeyerNest06}}] \label{KK-vanish}
Let $G$ be a countable, discrete, amenable and torsion-free group.
If $A$ is a separable \cstar-algebra satisfying $KK(A,A)=0$, then $KK^G(A,A)=0$ with respect to every $G$-action on $A$.
\end{theorem}

This has the following consequence:

\begin{prop} \label{homotopic-actions}
Let $G$ be a countable, amenable and torsion-free group.
Let $A$ be a \cstar-algebra and $\set{ \alpha^{(t)}: G\curvearrowright A }_{t\in [0,1]}$ a continuous family of $G$-actions on $A$.
Let $A[0,1]=\CC[0,1]\otimes A$ and $\alpha: G\curvearrowright A[0,1]$ the obvious action induced by this family. 
Then for $i=0,1$, the evaluation map $\ev_i: A[0,1]\to A$ induces a $KK^G$-equivalence between $\big( A[0,1], \alpha\big)$ and $(A,\alpha^{(i)})$.
In particular, $(A,\alpha^{(0)})$ and $(A,\alpha^{(1)})$ are $KK^G$-equivalent via $KK(\ev_0)\otimes KK(\ev_1)^{-1}$. 
\end{prop}
\begin{proof}
It clearly suffices to consider the case $i=0$. Denote $A(0,1] = C_0(0,1]\otimes A$ and by slight abuse of notation, we also write $\alpha$ for the restriction of $\alpha$ on this \cstar-subalgebra.
We then have a short exact sequence
\[
\xymatrix{
0 \ar[r] & \big( A(0,1], \alpha \big) \ar[r] & \big( A[0,1], \alpha \big) \ar[r]^{\ev_0} & (A,\alpha) \ar[r] & 0
}
\]
of $G$-\cstar-dynamical systems. Since $A(0,1]$ is a contractible \cstar-algebra, it follows from Theorem \ref{KK-vanish} that $\big( A(0,1], \alpha \big)$ has vanishing $KK^G$-class.
Using the induced six-term exact sequence in $KK^G$, we see that $\ev_0$ must indeed induce a $KK^G$-equivalence.
\end{proof}

\begin{lemma} \label{D-absorbing-KKG}
Let $G$ be a countable, amenable and torsion-free group.
Let $\CD$ be a strongly self-absorbing \cstar-algebra and $\alpha: G\curvearrowright\CD$ an action.
Then the equivariant first-factor embedding 
\[
\id_\CD\otimes\eins: (\CD,\id_\CD)\to (\CD\otimes\CD,\id_\CD\otimes\alpha)
\]
induces a $KK^G$-equivalence.
\end{lemma}
\begin{proof}
This is another straightforward consequence of \cite[Theorem 8.5]{MeyerNest06} as $\id_\CD\otimes\eins_\CD: \CD\to\CD\otimes\CD$ is a $KK$-equivalence.
\end{proof}

\begin{prop} \label{D-stable-homotopic}
Let $G$ be a topological group. Let $A$ be a separable \cstar-algebra and $\CD$ a strongly self-absorbing \cstar-algebra. Suppose that $A\cong A\otimes\CD$.
Let $\alpha: G\curvearrowright A$ be an action.
Then there is a continuous family $\set{ \alpha^{(t)}: G\curvearrowright A }_{t\in [0,1]}$ of actions such that $\alpha^{(0)}=\alpha$ and $\alpha^{(1)}$ is conjugate to $\alpha\otimes\id_\CD$.
\end{prop}
\begin{proof}
Let $\phi: A\to A\otimes\CD$ be an isomorphism and $v: [0,1)\to\CU(\widetilde{A\otimes\CD})$ a continuous map such that $\ad(v_t)\circ\phi \stackrel{t\to 1}{\longrightarrow} \id_A\otimes\eins$ in point-norm. Then
\[
\phi^{-1}\circ\ad(v_t^*)\circ(\psi\otimes\id_\CD)\circ\ad(v_t)\circ\phi \stackrel{t\to 1}{\longrightarrow} \psi
\]
in point-norm for every endomorphism $\psi: A\to A$. In particular, the continuous family of actions $\alpha^{(t)}$ given by 
\[
\alpha^{(t)}_g=\begin{cases} \phi^{-1}\circ\ad(v_{1-t}^*)\circ(\alpha_g\otimes\id_\CD)\circ\ad(v_{1-t})\circ\phi &,\quad t>0 \\
\alpha_g &,\quad t=0
\end{cases}
\]
is as desired. 
\end{proof}

\begin{prop} \label{ssa-all-homotopic}
Let $G$ be a topological group. Let $\CD$ be a strongly self-absorbing \cstar-algebra.
Let $\alpha: G\curvearrowright\CD$ be an action. Then there is a continuous family $\set{ \alpha^{(t)}: G\curvearrowright\CD }_{t\in [0,1]}$ of actions such that $\alpha^{(0)}=\alpha$ and $\alpha^{(1)}=\id_\CD$.
\end{prop}
\begin{proof}
This follows directly from Proposition \ref{D-stable-homotopic} and \cite[Proposition 4.5]{HirshbergPhillips15}.
\end{proof}

\begin{cor} \label{tf-D-stable}
Let $G$ be a countable, discrete, amenable and torsion-free group.
Let $\CD$ be a strongly self-absorbing Kirchberg algebra and $\alpha: G\curvearrowright\CD$ a pointwise outer action.
Then $\CD\rtimes_\alpha G$ is $\CD$-stable. Moreover, the unital embedding
\[
\id_\CD\otimes j_\alpha: \CD\otimes \mathrm{C}^*(G)\to\CD\otimes (\CD\rtimes_\alpha G)
\]
induces a $KK$-equivalence. 
\end{cor}
\begin{proof}
Upon combining Propositions \ref{ssa-all-homotopic} and \ref{homotopic-actions}, we see that $\CD\rtimes_\alpha G$ is $KK$-equivalent to the $\CD$-stable \cstar-algebra $\CD\otimes\cstar(G)$.
Then the first part of the claim follows from Kirchberg-Phillips classification. 

The second part of the claim follows directly from Lemma \ref{D-absorbing-KKG} and applying the descent to $KK$-theory, with obvious identifications of the involved crossed products.
\end{proof}

\begin{cor} \label{approx-rep-tf-uniqueness}
Let $G$ be a countable, discrete, torsion-free amenable group.
Let $\CD$ be a strongly self-absorbing Kirchberg algebra.
Let $\gamma$ be a model action as in Example \ref{the-model-Oinf}.
Then for any pointwise outer action $\alpha: G\curvearrowright\CD$, one has $\alpha\scc\gamma\otimes\id_\CD$ if and only if $\alpha$ is approximately representable.
\end{cor}
\begin{proof}
This follows directly from Corollary \ref{tf-D-stable} and Theorem \ref{approx-rep-ssa-uniqueness}.
\end{proof}


\section{Uniqueness of $\CO_2$-absorbing outer actions}

In this section, we show that pointwise outer actions of amenable groups on $\CO_2$ that equivariantly absorb the trivial action on $\CO_2$ are unique up to strong cocycle conjugacy.
For this purpose, we need to observe a few general facts about certain saturated and pointwise outer actions on simple and purely infinite \cstar-algebras.
The main ingredients come from Section 2, with the only addition here being that we specialize to cases where the fixed point algebras of the actions in question are in Cuntz standard form.

\begin{lemma} \label{O2-morph-uniqueness-1}
Let $G$ be a countable, discrete and amenable group. Let $D$ be a unital, simple and purely infinite \cstar-algebra.
Let $\alpha: G\curvearrowright D$ be a saturated action. Assume that for every separable, $\alpha$-invariant \cstar-subalgebra $A\subset D$, the induced action of $\alpha$ on $D\cap A'$ is pointwise non-trivial and the fixed point algebra $(D\cap A')^\alpha$ contains a unital copy of $\CO_2$.

Let $B$ be a separable, unital, simple and nuclear \cstar-algebra with an action $\beta: G\curvearrowright B$. Then any two unital and equivariant $*$-homomorphisms from $(B,\beta)$ to $(D,\alpha)$ are $G$-unitarily equivalent.
\end{lemma}
\begin{proof}
Let us first observe that it follows from our assumptions on $(D,\alpha)$ that the induced action on any relative commutant of the above form is actually pointwise outer:
If the restriction of $\alpha_g$, for some $g\neq 1$, on $D\cap A'$ were induced by a unitary $w$, then the restriction of $\alpha_g$ on $D\cap A_0'$, for $A_0=\cstar\big( A\cup\set{\alpha_g(w)}_{g\in G} \big)$, would be trivial, which contradicts our assumption.

Let 
\[
\phi_1,\phi_2: (B,\beta)\to (D,\alpha)
\]
be two unital and equivariant $*$-homomorphisms. Consider the unital and equivariant $*$-homomorphism
\[
\pi: (B,\beta)\to (M_2\otimes D, \id_{M_2}\otimes\alpha),\quad b\mapsto\matrix{ \phi_1(b) & 0 \\ 0 & \phi_2(b) }.
\]
By Proposition \ref{semi-sat-permanence}\ref{semi-sat-permanence:6} and Corollary \ref{spi-sat}, the action on the target is saturated.
It also follows from our assumptions on $\alpha$ that the action induced on the relative commutant $(M_2\otimes D)\cap\pi(B)'$ is pointwise outer.
This \cstar-algebra is moreover simple and purely infinite by Theorem \ref{F() pis}. It thus follows from Lemma \ref{fixed point} that the fixed point algebra
\[
C:=\Big( (M_2\otimes D)\cap\pi(B)' \Big)^{\id_{M_2}\!\otimes\alpha}
\]
is simple and purely infinite. By our assumptions on $\alpha$, there exists a unital copy of $\CO_2$ in the \cstar-algebra $\Big( D\cap \big(\phi_1(B)\cup\phi_2(B) \big)' \Big)^\alpha$. 
Applying the diagonal embedding from $D$ to $M_2\otimes D$ to this yields a unital copy of $\CO_2$ in $C$ that also commutes with the projections $e_{1,1}\otimes\eins, e_{2,2}\otimes\eins\in C\subset M_2\otimes D$. 
It follows that these projections have trivial $K$-theory class. Since $C$ is simple and purely infinite, it follows from \cite{Cuntz81} that $e_{1,1}\otimes\eins$ and $e_{2,2}\otimes\eins$ are Murray-von-Neumann equivalent in $C$. 
Letting $r\in C$ be a partial isometry with $r^*r=e_{1,1}\otimes\eins$ and $rr^*=e_{2,2}\otimes\eins$, we see that $r=e_{2,1}\otimes v$ for a unitary $v\in D^\alpha$ satisfying $\phi_2(b)=v\phi_1(b)v^*$ for all $b\in B$. This shows our claim.
\end{proof}

\begin{lemma} \label{O2-morph-uniqueness-2}
Let $G$ be a countable, discrete and amenable group. Let $A$ be a Kirchberg algebra and $\alpha: G\curvearrowright A$ a pointwise outer action that is equivariantly $\CO_2$-absorbing. Let $\omega\in\beta\IN\setminus\IN$ be a free ultrafilter. Let $S\subset A_\omega$ be a separable, unital, simple, nuclear and $\alpha_\omega$-invariant \cstar-algebra.
Then the induced action $\tilde{\alpha}_\omega: G\curvearrowright F(S,A_\omega)$ satisfies the conditions in Lemma \ref{O2-morph-uniqueness-1} in place of $(D,\alpha)$.

In particular, if $\beta: G\curvearrowright B$ is an action on a separable, unital, simple, nuclear \cstar-algebra, then any two unital and equivariant $*$-homomorphisms from $(B,\beta)$ to $\big( F(S,A_\omega), \tilde{\alpha}_\omega \big)$ are $G$-unitarily equivalent.
\end{lemma}
\begin{proof}
It follows from Theorem \ref{F() pis} that $F(S,A_\omega)$ is simple and purely infinite. 
The action $\tilde{\alpha}_\omega$ induced on it is saturated by Example \ref{ex:ultra}, Proposition \ref{F(A)-saturated} and Corollary \ref{spi-sat}. 
By Theorem \ref{central seq outer}, the induced action $\tilde{\alpha}_\omega: G\curvearrowright F_\omega(A)$ is pointwise non-trivial. 
Since $S$ is separable, it then follows from a standard reindexation argument that the restriction of $\tilde{\alpha}_\omega$ to any relative commutant of a separable, $\tilde{\alpha}_\omega$-invariant \cstar-subalgebra of $F(S,A_\omega)$ is also pointwise non-trivial.
By \cite[Theorem 3.7]{Szabo16ssa}, the fact that $\alpha$ is equivariantly $\CO_2$-absorbing is equivalent to the existence of a unital copy of $\CO_2$ in the fixed point algebra $F_\omega(A)^{\tilde{\alpha}_\omega}$.
Again, a standard reindexation argument, cf.\ \cite[Lemma 1.14]{Szabo16ssa2}, implies that a unital copy of $\CO_2$ can be found in every fixed point algebra of a relative commutant of a separable, $\tilde{\alpha}_\omega$-invariant \cstar-subalgebra in $F(S,A_\omega)$.
The rest follows from Lemma \ref{O2-morph-uniqueness-1}.
\end{proof}

\begin{prop} \label{O2-morph-ex-1}
Let $G$ be a countable, discrete and amenable group.
Let $D$ be a unital, simple and purely infinite \cstar-algebra and $\alpha: G\curvearrowright D$ a saturated and pointwise outer action.
Assume that the fixed point algebra $D^\alpha$ is in Cuntz standard form, i.e., it contains a unital copy of $\CO_2$.
Then for every action $\beta: G\curvearrowright B$ on a separable, unital, exact \cstar-algebra, there exists a unital and equivariant $*$-homomorphism from $(B,\beta)$ to $(D,\alpha)$.
\end{prop}
\begin{proof}
Since $B$ is exact and $G$ is amenable, the crossed product $B\rtimes_\beta G$ is again separable, unital and exact.
Let $\iota: B\to B\rtimes_\beta G$ be the inclusion. By the Kirchberg-Phillips $\CO_2$-embedding theorem \cite{KirchbergPhillips00}, we find a unital embedding $\kappa: B\rtimes_\beta G\to\CO_2$.
Consider the unitary representation given by $w_g=\kappa(\lambda^\beta_g)\in\CU(\CO_2)$ for all $g\in G$.
Because $D^\alpha$ contains a unital copy of $\CO_2$, it follows from Lemma \ref{O embeddings} that there exists a unital and equivariant $*$-homomorphism $\psi: (\CO_2,\ad(w))\to (D,\alpha)$.
Then $\phi=\psi\circ\kappa\circ\iota$ defines a unital and equivariant $*$-homomorphism from $(B,\beta)$ to $(D,\alpha)$ because one has for all $g\in G$ and $b\in B$ that
\[
(\alpha_g\circ\phi)(b)=\psi(\ad(w_g)(\kappa(\iota(b)))) = (\psi\circ\kappa)(\ad(\lambda^\beta_g)(b)) = (\phi\circ\beta_g)(b).
\]
\end{proof}

\begin{cor} \label{O2-morph-ex-2}
Let $G$ be a countable, discrete and amenable group.
Let $A$ be a Kirchberg algebra and $\alpha: G\curvearrowright A$ a pointwise outer action that is equivariantly $\CO_2$-absorbing.
Let $\omega\in\beta\IN\setminus\IN$ be a free ultrafilter.
Then for every action $\beta: G\curvearrowright B$ on a separable, unital, exact \cstar-algebra, there exists a unital and equivariant $*$-homomorphism from $(B,\beta)$ to $\big( F_\omega(A), \tilde{\alpha}_\omega \big)$.
\end{cor}

\begin{theorem} \label{O2-absorbing-uniqueness}
Let $G$ be a countable, discrete and amenable group.
Let $\alpha: G\curvearrowright \CO_2$ be a pointwise outer and equivariantly $\CO_2$-absorbing action. Then $\beta\otimes\alpha\scc\alpha$ for every action $\beta: G\curvearrowright B$ on a separable, unital, simple and nuclear \cstar-algebra.

In particular, there exists, up to strong cocycle conjugacy, precisely one pointwise outer $G$-action on $\CO_2$ that is equivariantly $\CO_2$-absorbing.
\end{theorem}
\begin{proof}
For notational convenience, we will denote in some instances $A=\CO_2$.
The proof is for the most part a straightforward modification of the classical proof of Kirchberg-Phillips, but using the equivariant intertwining methods from \cite[Section 2]{Szabo16ssa} instead of classical intertwining.

Let us show the first part of the statement. Let $\beta: G\curvearrowright B$ be an action on a separable, unital, simple and nuclear \cstar-algebra.

By Corollary \ref{O2-morph-ex-2}, there exists an equivariant, unital $*$-homomor\-phism from $(B,\beta)$ to $\big( A_\omega\cap A', \alpha_\omega \big)$. Consider the canonical inclusions 
\[
A_\omega\cap A' ,~ B ~\subset~ (B\otimes A)_\omega\cap (\eins_B\otimes A)'
\]
which define commuting \cstar-subalgebras. Since these inclusions are natural, they are equivariant with respect to the induced actions of $\alpha$, $\beta$ and $\beta\otimes\alpha$.

By assumption, it follows that we have a unital and equivariant $*$-homo\-morphism
\[
\phi: (B\otimes B,\beta\otimes\beta)\to \big( (B\otimes A)_\omega)\cap(\eins_B\otimes A)' , (\beta\otimes\alpha)_\omega \big)
\]
satisfying $\phi(b\otimes\eins_B)\cdot (\eins_B\otimes a)=b\otimes a$ and $\phi(\eins_B\otimes b)\in\eins_{B}\otimes (A_\omega\cap A')$ for all $a\in A$ and $b\in B$.

Since we assumed that $\alpha$ is equivariantly $\CO_2$-absorbing, the action $\beta\otimes\alpha: G\curvearrowright B\otimes A$ is also pointwise outer and equivariantly $\CO_2$-absorbing.
It then follows from Lemma \ref{O2-morph-uniqueness-2} that $\phi\circ (\id_B\otimes\eins)$ and $\phi\circ (\eins\otimes\id_B)$ are $G$-unitarily equivalent.
Let $v\in \Big( (B\otimes A)_\omega\cap(\eins_B\otimes A)' \Big)^{(\beta\otimes\alpha)_\omega}$ be a unitary with 
\[
\phi(b\otimes\eins_B)=v\phi(\eins\otimes b)v^*
\] 
for all $b\in B$. Then one has
\[
v^*(b\otimes a)v = v^*\phi(b\otimes\eins_B)v^*\cdot (\eins_B\otimes a) = \phi(\eins_B\otimes b)\cdot (\eins\otimes a)\in \eins_B\otimes A_\omega
\] 
for all $a\in A$ and $b\in B$. 
Let $v_n\in B\otimes A$ be a sequence of unitaries representing $v$. It then follows that
\begin{itemize}
\item $\dst\lim_{n\to\omega} \|[v_n, \eins_B\otimes a]\|=0$ for all $a\in A$;
\item $\dst\lim_{n\to\omega} \dist\big( v_n^*(b\otimes a)v_n, \eins_B\otimes A \big) = 0$ for all $b\in B$ and $a\in A$;
\item $\dst\lim_{n\to\omega} \| v_n-(\beta\otimes\alpha)_g(v_n) \| =0$ for all $a\in A$ and $g\in G$.
\end{itemize} 
This implies that the equivariant second-factor embedding
\[
\eins_B\otimes\id_A : (A,\alpha)\to (B\otimes A, \beta\otimes\alpha)
\]
satisfies the requirements of \cite[Lemma 2.1]{Szabo16ssa}. It follows that $\alpha$ and $\beta\otimes\alpha$ are strongly cocycle conjugate, showing the first part.

For the second part, let $\alpha^{(i)}: G\curvearrowright\CO_2$ be two pointwise outer, equivariantly $\CO_2$-absorbing actions for $i=1,2$. Applying the first part twice, we obtain
\[
\alpha^{(1)}\scc\alpha^{(1)}\otimes\alpha^{(2)}\scc\alpha^{(2)}.
\]
This finishes the proof.
\end{proof}

\begin{example} \label{the-model-O2}
As in Example \ref{the-model-Oinf}, let us choose a unitary representation $v: G\to\CU(\CO_2)$ that induces a $*$-monomorphism on $\cstar(G)$. We can then consider
\[
\delta=\bigotimes_\IN\ad(v): G\curvearrowright\bigotimes_\IN\CO_2\cong\CO_2,
\]
which for our purposes can be regarded as {\it the} model of a pointwise outer and equivariantly $\CO_2$-absorbing action.
Indeed, the outerness comes from the infinite tensor product form involving a faithful action, and upon observing that a unitary representation is always a cocycle for the trivial action, we see $\delta\cc\delta\otimes\id_{\CO_2}$.
Thus we recover the main result of this section in the way that it is stated in Theorem \ref{O2-uniqueness-intro}.
Notice that this automatically verifies the claim about $\delta$ in Theorem \ref{choice-independence-intro}.
\end{example}

\begin{rem}
We remark that Proposition \ref{KK-or-KL}\ref{KK-or-KL:2} further shows that the uniqueness in Theorem \ref{O2-absorbing-uniqueness} can be improved with respect to very strong cocycle conjugacy.
\end{rem}

\begin{rem}
Just as the $K$-theoretic condition in Theorem \ref{abstract-models-Oinf} is a necessary assumption in general, the $\CO_2$-absorption does generally not come for free in the context of Theorem \ref{O2-absorbing-uniqueness}.
Basic examples of outer actions of $\IZ_2$ on $\CO_2$ with non-trivial $K$-theory in the crossed product are well-known; see \cite{Izumi04}.
However, Matui's uniqueness result \cite{Matui08} for $\IZ^N$-actions on $\CO_2$ clearly showcases the redundancy of this assumption for $\IZ^N$, and this phenomenon is in fact observed for poly-$\IZ$ groups in ongoing work of Izumi-Matui.
Since every action of a torsion-free amenable group on $\CO_2$ has trivial $KK^G$-class, this leads one to think that all pointwise outer actions of such groups on $\CO_2$ ought to be equivariantly $\CO_2$-absorbing.
It seems unclear whether or not this problem is just as difficult or easier to understand than the question about the redundancy of approximate representability posed in Remark \ref{conj:ar-redundant}.
\end{rem}


\section{A uniqueness theorem for homotopic maps}

In this section, we prove a uniqueness theorem for equivariant $*$-homom\-orphisms with respect to homotopy, leading to Theorem \ref{homotopy-rigid-intro}.

\begin{defi} \label{equivariant-homotopy}
Let $G$ be a countable, discrete group.
Let $\alpha: G\curvearrowright A$ and $\beta: G\curvearrowright B$ be two actions on \cstar-algebras.
We say that two equivariant $*$-homomorphisms $\phi_0, \phi_1: (A,\alpha)\to (B,\beta)$ are $G$-homotopic, if there exists an equivariant $*$-homomorphism $\Phi: (A,\alpha)\to\big( B[0,1], \beta \big)$ such that $\ev_i\circ\Phi = \phi_i$ for $i=0,1$.

We say that $(A,\alpha)$ and $(B,\beta)$ are $G$-homotopy equivalent, if there exist equivariant $*$-homomorphisms $\phi: (A,\alpha)\to (B,\beta)$ and $\psi: (B,\beta)\to (A,\alpha)$ such that $\psi\circ\phi$ is $G$-homotopic to $\id_A$ and $\phi\circ\psi$ is $G$-homotopic to $\id_B$.
\end{defi}

\begin{theorem} \label{O2-abs-map-uniqueness}
Let $G$ be a countable, discrete, amenable group.
Let $A$ be a separable, simple, unital, nuclear \cstar-algebra and let $B$ be a unital Kirchberg algebra.
Consider two actions $\alpha: G\curvearrowright A$ and $\beta: G\curvearrowright B$.
Suppose that $\beta$ is pointwise outer.
Let $\phi, \psi: (A,\alpha)\to (B,\beta)$ be two unital and equivariant $*$-homomorphisms.
Then $\phi\otimes\eins_{\CO_2}$ and $\psi\otimes\eins_{\CO_2}$ are approximately $G$-unitarily equivalent as equivariant $*$-homomorphisms from $(A,\alpha)$ to $(B\otimes\CO_2, \beta\otimes\id_{\CO_2})$.
\end{theorem}
\begin{proof}
This follows directly from Lemma \ref{O2-morph-uniqueness-1} applied to $(B\otimes\CO_2)_\omega$ in place of $D$ and $(\beta\otimes\id_{\CO_2})_\omega$ in place of $\alpha$, where $\omega$ is any free ultrafilter on $\IN$.
\end{proof}

The following observation is contained in \cite[Lemma 4.14 and Remark 4.16]{Szabo17Rf} and goes by adapting an argument of Phillips in the proof of \cite[Theorem 3.1]{Phillips97}.
Although this holds in a much more general context, we will only state it for discrete groups here.

\begin{lemma} \label{homotopy-uniqueness-Phillips}
Let $G$ be a countable, discrete group.
Let $A$ and $B$ be \cstar-algebras with actions $\alpha: G\curvearrowright A$ and $\beta: G\curvearrowright B$.
Suppose that $\phi_0, \phi_1: (A,\alpha)\to (B,\beta)$ are equivariant $*$-homomorphisms that are $G$-homotopic through $\Phi: (A,\alpha)\to (B[0,1],\beta)$.
Suppose that for all $s,t\in [0,1]$, the equivariant $*$-homomorphisms $\Phi_t\otimes\eins_{\CO_2}$ and $\Phi_s\otimes\eins_{\CO_2}$ are approximately $G$-unitarily equivalent.
Then $\phi_0\otimes\eins_{\CO_\infty}$ and $\phi_1\otimes\eins_{\CO_\infty}$ are approximately $G$-unitarily equivalent.
\end{lemma}

\begin{theorem} \label{homotopy-map-uniqueness}
Let $G$ be a countable, discrete, amenable group.
Let $A$ be a separable, simple, unital, nuclear \cstar-algebra and let $B$ be a unital Kirchberg algebra.
Consider two actions $\alpha: G\curvearrowright A$ and $\beta: G\curvearrowright B$.
Suppose that $\beta$ is pointwise outer.
Let $\phi_0, \phi_1: (A,\alpha)\to (B,\beta)$ be two unital and equivariant $*$-homomorphisms.
If $\phi_0$ and $\phi_1$ are $G$-homotopic, then they are approximately $G$-unitarily equivalent.
\end{theorem}
\begin{proof}
By Theorem \ref{equ-Oinf-absorption}, we have $\beta\cc\beta\otimes\id_{\CO_\infty}$.
In particular, if we fix a free ultrafilter $\omega$ on $\IN$, there is a commutative diagram of equivariant $*$-homomorphisms of the form
\[
\xymatrix{
(B,\beta) \ar[dr] \ar[rr] && (B_\omega,\beta_\omega) \\
& (B\otimes\CO_\infty, \beta\otimes\id_{\CO_\infty}) \ar[ur] &
}
\]
It thus suffices to show that $\phi_0\otimes\eins_{\CO_\infty}$ and $\phi_1\otimes\eins_{\CO_\infty}$ are approximately $G$-unitarily equivalent as equivariant $*$-homomorphisms from $(A,\alpha)$ to $(B\otimes\CO_\infty,\beta\otimes\id_{\CO_\infty})$.

The claim then follows directly from Lemma \ref{homotopy-uniqueness-Phillips} and Theorem \ref{O2-abs-map-uniqueness}.
\end{proof}

The following confirms Theorem \ref{homotopy-rigid-intro}:

\begin{cor} \label{homotopy-rigidity}
Let $G$ be a countable, discrete, amenable group.
Let $A$ and $B$ be two unital Kirchberg algebras with pointwise outer actions $\alpha: G\curvearrowright A$ and $\beta: G\curvearrowright B$.
If $(A,\alpha)$ and $(B,\beta)$ are $G$-homotopy equivalent, then they are conjugate.
\end{cor}
\begin{proof}
This follows directly from Theorem \ref{homotopy-map-uniqueness} and \cite[Corollary 1.16]{Szabo16ssa}.
\end{proof}

Next, we shall apply Theorem \ref{homotopy-map-uniqueness} to some non-trivial examples of actions on $\CO_\infty$.

\begin{defi}
Let $G$ be a countable, discrete group.
Let $I$ be a countable index set and $\sigma: G\curvearrowright I$ an action.
Let $\CD$ be a strongly self-absorbing Kirchberg algebra.
We consider $\beta^\sigma: G\curvearrowright\bigotimes_I \CD$ the noncommutative Bernoulli shift induced by $\sigma$, which is given by
\[
\beta^\sigma_g\big( \bigotimes_{i\in I} d_i \big) = \bigotimes_{i\in I} d_{\sigma_g(i)}.
\]
Note that when $I$ is infinite, then all but finitely many of the elements $d_i\in\CD$ are assumed to be equal to $\eins_\CD$.
\end{defi}

\begin{cor} \label{bern-ssa}
Let $G$ be a countable, discrete, amenable group.
Let $\CD$ be a strongly self-absorbing Kirchberg algebra.
Then all noncommutative Bernoulli shift actions of $G$ on tensor products of $\CD$ are strongly self-absorbing.
\end{cor}
\begin{proof}
Let $I$ be a countable index set and $\sigma: G\curvearrowright I$ an action.
Consider the associated Bernoulli shift $\beta^\sigma: G\curvearrowright\bigotimes_I \CD$.
Since $\CD$ is strongly self-absorbing, one has $\CD\cong\bigotimes_\IN\CD$, and thus a straightforward rearrangement of tensor factors shows that $\beta^\sigma$ is conjugate to $\bigotimes_\IN \beta^\sigma$.

In particular, an automorphism of the form $\beta^\sigma_g$ for $g\in G$ is either trivial or outer.
Thus, by dividing out the kernel of $\sigma$, we may assume without loss of generality that it is faithful, so $\beta^\sigma$ is pointwise outer.

Using again that $\CD$ is strongly self-absorbing, the flip automorphism on $\CD\otimes\CD$ is homotopic to the identity map; see \cite{DadarlatWinter09}. 
Applying this homotopy on each individual opposed pair of tensor factors in the product $(\bigotimes_I\CD)\otimes (\bigotimes_I\CD)$, we obtain a $G$-homotopy between the identity map and the flip automorphism with respect to the action $\beta^\sigma\otimes\beta^\sigma$; see \cite[Lemma 2.6]{KerrLupiniPhillips15} for a related argument.
Thus it follows from Theorem \ref{homotopy-map-uniqueness} that the flip automorphism on $\big( (\bigotimes_I\CD)\otimes (\bigotimes_I\CD), \beta^\sigma\otimes\beta^\sigma \big)$ is approximately $G$-inner.

The claim now follows from \cite[Proposition 3.4]{Szabo16ssa}.
\end{proof}

Combining this fact with another $K$-theoretic computation using Baum-Connes, we will show that our model action on $\CO_\infty$ from Example \ref{the-model-Oinf} can be realized as any faithful noncommutative Bernoulli shift.
For this purpose, we need to make some observations.

\begin{lemma} \label{bern-izumi}
Let $\Gamma$ be a finite group.
Let $A$ and $B$ be two separable, nuclear \cstar-algebras.
Let $I$ be a finite index set and $\sigma: \Gamma\curvearrowright I$ an action.
Suppose that $\phi: A\to B$ is a $*$-homomorphism inducing a $KK$-equivalence.
Then
\[
\phi^{\otimes I}: A^{\otimes I}\to B^{\otimes I}
\]
induces a $KK^\Gamma$-equivalence, where we equip both $A^{\otimes I}$ and $B^{\otimes I}$ with the noncommutative Bernoulli shift induced by $\sigma$. 
\end{lemma}
\begin{proof}
This follows from a straightforward generalization of the proof of \cite[Theorem 2.1]{Izumi17}, which essentially treats the case $\Gamma=\IZ_2$ and $\sigma$ being the flip on the two-point set.
\end{proof}

\begin{cor} \label{bern-O2-Oinf}
Let $G$ be a countable, discrete, amenable group.
Let $I$ be a countable index set and $\sigma: G\curvearrowright I$ an action.
If we equip $\CO_\infty \cong \CO_\infty^{\otimes I}$ with the noncommutative Bernoulli shift $\beta^\sigma$ induced by $\sigma$, then the unital inclusion $(\IC,\id)\subset (\CO_\infty,\beta^\sigma)$ induces a $KK^G$-equivalence.
\end{cor}
\begin{proof}
Denote by $\iota: \IC\to\CO_\infty$ the unital inclusion.
Then the unital inclusion $(\IC,\id)\subset (\CO_\infty,\beta^\sigma)$ may naturally be identified with $\iota^{\otimes I}$.
Using Baum-Connes in the form of \cite[Theorem 8.5]{MeyerNest06}, we only need to show that $\iota^{\otimes I}$ induces a $KK^\Gamma$-equivalence for all finite subgroups $\Gamma\subseteq G$, where we restrict the $G$-\cstar-algebras to $\Gamma$-\cstar-algebras.
Note that these restrictions are also noncommutative Bernoulli shifts of $\Gamma$.
Thus, our claim reduces to the fact where $G=\Gamma$ is a finite group.

If $I$ is finite, then this is a special case of Lemma \ref{bern-izumi}.

If $I$ is infinite, then write $I$ as a disjoint union $I=\bigsqcup_{n\in\IN} I_n$ so that each $I_n$ is finite and $\sigma$-invariant.
Then we have a tensor product decomposition
\[
\CO_\infty^{\otimes I} = \bigotimes_{n\in\IN} \CO_\infty^{\otimes I_n},
\]
which is compatible with respect to the Bernoulli subshifts.
Consider the following commutative diagram:
\[
\xymatrix@R+5mm@C+5mm{
\CO_\infty^{\otimes I_1} \ar[r] & \CO_\infty^{\otimes I_1}\otimes\CO_\infty^{\otimes I_2} \ar[r] & \cdots \ar[r] & \bigotimes_{n\leq m} \CO_\infty^{\otimes I_n} \ar[r] & \CO_\infty^{\otimes I} \\
\IC \ar[u]^{\iota^{\otimes I_1}} \ar[ur]^{\iota^{\otimes I_1}\otimes\iota^{\otimes I_2}} \ar[urrr] \ar[urrrr]_{\iota^{\otimes I}} &&&
}
\]
We know that all intermediate horizontal maps are $KK^\Gamma$-equivalences as well as all maps of the form $\bigotimes_{n\leq m} \iota^{\otimes I_n}$ from the bottom to one of the building blocks.
Thus $\phi^{\otimes I}$ itself must be a $KK^\Gamma$-equivalence.\footnote{This is analogous to \cite[Proposition 2.4]{Dadarlat09}. One uses that the diagram implies the vanishing of $\dst {\lim_{\longleftarrow}}^1 KK^\Gamma\big( \_, \bigotimes_{n\leq m} \CO_\infty^{\otimes I_n} \big)$, which one combines with the Milnor sequence \cite[Theorem 21.5.2]{BlaKK} for equivariant $KK$-theory.}
\end{proof}

\begin{cor} \label{bern-models}
Let $G$ be a countable, discrete, amenable group. 
Then every faithful Bernoulli $G$-shift on $\CO_\infty$ is strongly cocycle conjugate to the model action $\gamma: G\curvearrowright\CO_\infty$ in Example \ref{the-model-Oinf}.
\end{cor}
\begin{proof}
This follows directly from combining Corollaries \ref{bern-ssa}, \ref{bern-O2-Oinf} and  \ref{abstract-models-Oinf}.
\end{proof}


\section{Finiteness of Rokhlin dimension}

In this section, we show that pointwise outer actions of amenable, residually finite groups on Kirchberg algebras have Rokhlin dimension at most one in the sense of \cite{HirshbergWinterZacharias15, SzaboWuZacharias15}. 
Note that this has been observed for symmetries in \cite{BarlakEndersMatuiSzaboWinter} and for finite groups in \cite{Gardella17_R}.
See also \cite{Liao16, Liao17} for results in this direction in the context of finite \cstar-algebras.

For the purpose of obtaining the aforementioned goal, we first examine the prevalence of Rokhlin towers in saturated and pointwise outer \cstar-dynamical systems on simple and purely infinite \cstar-algebras;
 the main effort is again already dealt with in Section 2 and we merely have to specialize the setting a little bit. 
 
Let us recall a useful well-known fact, which was implicitely used for example in \cite{BarlakEndersMatuiSzaboWinter} and \cite{MatuiSato14UHF}.

\begin{prop}[see {\cite[Lemma 1.2]{Szabo16spi}}] \label{full spectrum}
Let $B$ be a unital, simple and purely infinite \cstar-algebra.
Then any two positive contractions in $B$ with full spectrum $[0,1]$ are approximately unitarily equivalent.
\end{prop}

\begin{lemma} \label{Rokhlin tower}
Let $G$ be a discrete, countable and amenable group and $H\subset G$ a subgroup with finite index.
Let $D$ be a unital, simple and purely infinite \cstar-algebra and let $\alpha: G\curvearrowright D$ be a saturated and pointwise outer action. Then 
\begin{enumerate}[label=\textup{(\arabic*)},leftmargin=*]
\item there exists a non-zero equivariant $*$-homomorphism 
\[
\phi: (\CC(G/H), G\text{-shift}) \to (D,\alpha).
\] \label{Rokhlin tower:1}
\item there exist two equivariant c.p.c.~order zero maps 
\[
\phi_0,\phi_1: (\CC(G/H), G\text{-shift})\to (D,\alpha)
\] 
with $\phi_0(\eins)+\phi_1(\eins)=\eins$. \label{Rokhlin tower:2}
\end{enumerate}
\end{lemma}
\begin{proof}
Since $H$ has finite index, it also contains some normal subgroup $N\subset G$ of finite index.
One then has a natural inclusion $(\CC(G/H), G\text{-shift})\subset(\CC(G/N), G\text{-shift})$.
In particular, we may assume without loss of generality that $H$ is normal to begin with.

\ref{Rokhlin tower:1}: Let us consider some inclusions
\[
\CC(G/H) \subset M_{|G/H|} \subset \CO_2 \stackrel{\id\oplus 0}{\subset} \CO_2\oplus\IC \subset \CO_\infty,
\]
where the first inclusion is the obvious one, and all of them are unital except for the middle one.
The $(G/H)$-shift on $\CC(G/H)$ is implemented by the left-regular representation of $G/H$ inside $M_{|G/H|}$.
By inducing the representation to $G$ and using the above chain of inclusions, we can thus find some $KK$-trivial unitary representation $u: G\to\CO_\infty$ and a non-zero equivariant $*$-homomorphism from $ (\CC(G/H), G\text{-shift})$ to $(\CO_\infty,\ad(u))$.
The claim then follows directly from Corollary \ref{model-sat-embedding}.

\ref{Rokhlin tower:2}: Let $\phi: (\CC(G/H), G\text{-shift}) \to (D,\alpha)$ be a non-zero $*$-homomorphism, which exists by \ref{Rokhlin tower:1}.
The \cstar-algebra $D^{\alpha|_H}$ is simple and purely infinite by Lemma \ref{fixed point}, and so is the corner $B=p(D^{\alpha|_H})p$ given by $p=\phi(e_1)$. 
Let us choose a positive contraction $h'\in B$ with full spectrum $[0,1]$.
Consider $h_0=\sum_{\bar{g}\in G/H} \alpha_g(h')$, which is then a well-defined positive contraction of $D^\alpha$ with full spectrum and satisfies $h_0\in D\cap\phi(\CC(G/H))'$ and $\phi(\eins)h_0=h_0$.
By Lemma \ref{fixed point} applied to $\alpha$, Proposition \ref{full spectrum} and saturation, it follows that there exists a unitary $v\in\CU(D^\alpha)$ with $vh_0v^*=\eins-h_0$.
Then $\phi_0=h_0\cdot\phi$ and $\phi_1=\ad(v)\circ\phi_0$ satisfy the desired properties.
\end{proof}

\begin{theorem}
Let $A$ be a Kirchberg algebra and $G$ a countable, discrete, amenable group and $H\subset G$ a subgroup with finite index.
Let $(\alpha,w): G\curvearrowright A$ be a pointwise outer cocycle action. Then $\dimrok(\alpha, H)\leq 1$.
In particular, if $G$ is moreover assumed to be residually finite, then $\dimrok(\alpha)\leq 1$.
\end{theorem}
\begin{proof}
It follows from Proposition \ref{F(A)-sat} and Theorem \ref{central seq outer} that $\tilde{\alpha}_\omega: G\curvearrowright F_\omega(A)$ is saturated, pointwise outer, and the underlying \cstar-algebra is simple and purely infinite.
Applying Lemma \ref{Rokhlin tower}\ref{Rokhlin tower:2}, there exist two equivariant c.p.c.\ order zero maps
\[
\phi_0,\phi_1: (\CC(G/H),G\text{-shift}) \to (F_\omega(A), \tilde{\alpha}_\omega)
\]
with $\phi_0(\eins)+\phi_1(\eins)=\eins$.
This directly shows $\dimrok(\alpha, H)\leq 1$, cf.\ \cite[Section 4]{SzaboWuZacharias15}. If $G$ is moreover residually finite, then it follows that 
\[
\dimrok(\alpha) \stackrel{\text{def}}{=} \sup\set{ \dimrok(\alpha,H) \mid H\leq G,\ [G:H]<\infty } \leq 1.
\]
\end{proof}


\bibliographystyle{gabor}
\bibliography{master}

\end{document}